\documentclass[preprint,12pt]{elsarticle}

\usepackage{mathrsfs}
\usepackage{booktabs}
\usepackage{multirow}
\usepackage{amssymb}
\usepackage{amsmath}
\usepackage{amsthm}
\usepackage{natbib}
\usepackage{hyperref}

\newtheorem{theorem}{Theorem}[section]

\newtheorem{example}[theorem]{Example}
\newtheorem{lemma}[theorem]{Lemma}

\theoremstyle{definition}
\newtheorem{definition}[theorem]{Definition}
\newtheorem{remark}[theorem]{Remark}
\DeclareMathOperator{\sinc}{sinc}

\journal{Journal of Mathematical Analysis and Applications}

\begin{document}

\begin{frontmatter}


\title{Convergence analysis of max-product and max-min Durrmeyer-type exponential sampling operators in Mellin Orlicz space}
\author{Satyaranjan Pradhan\fnref{label1}}\ead{satya.math1993@gmail.com} %% Author name
\affiliation[label1]{organization={Department of Mathematics},%Department and Organization
            addressline={Berhampur University}, 
            city={Berhampur},
            postcode={Bhanja Bihar 760007}, 
            state={Odisha},
            country={India}}
\author{H.M. Srivastava \corref{cor1}\fnref{label2,label3,label4,label5,label6,label7}}
 \ead{harimsri@math.uvic.ca, harimsri@uvic.ca}
 \cortext[cor1]{Corresponding author:}
 \affiliation[label2]{organization={Department of Mathematics and Statistics},
             addressline={University of Victoria},
           city={Victoria},
            state={British Columbia}, postcode={ V8W 3R4},
           country={Canada}}
\affiliation[label3]{organization={Department of Medical Research},
             addressline={China Medical University Hospital},
             city={China Medical University},
             postcode={Taichung 40402},
             country={Taiwan}}
\affiliation[label4]{organization={Center for Converging Humanities},
             addressline={Kyung Hee University},
             city={ 26 Kyungheedae-ro, Dongdaemun-gu, Seoul 02447},
             country={Republic of Korea}}  
 \affiliation[label5]{organization={Department of Applied Mathematics},
             addressline={Chung Yuan Christian University},
             city={Chung-Li, Taoyuan City 320314},
             country={Taiwan}}
 \affiliation[label6]{organization={Department of Mathematics and Informatics},
             addressline={Azerbaijan University},
             city={71 Jeyhun Hajibeyli Street},
             postcode={AZ1007 Baku},
             country={Azerbaijan}}
\affiliation[label7]{organization={Section of Mathematics},
             addressline={International Telematic University Uninettuno},
             city={39 Corso Vittorio Emanuele II, I-00186 Rome},
           country={Italy}}
\author{Madan Mohan Soren\corref{cor1}\fnref{label1}} %% Author name
\ead{mms.math@buodisha.edu.in}

%% Abstract
\begin{abstract}
In the present study, we establish both pointwise and uniform convergence in the space of logarithmically uniformly continuous and bounded functions for the max-product and max-min Durrmeyer-type exponential sampling operators. Furthermore, the modular convergence of these operators is demonstrated within the framework of Orlicz space. In addition to the theoretical results, we provide numerical and graphical analyses for various kernel pairs, illustrating the convergence rates and approximation behavior of the proposed operators.
\end{abstract}



%% Keywords
\begin{keyword} Approximation \sep convergence \sep Durrmeyer-type exponential sampling \sep exponential sampling \sep max-product and max-min operators \sep Orlicz space.


 \MSC[2020] 41A25 \sep 41A35 \sep 41A36 \sep 47A58.

\end{keyword}

\end{frontmatter}



\section{Introduction}\label{section1}
Approximation theory plays a vital role in signal analysis and image processing \cite{benedetto2001appli,costarelli2023appli}, where one of its central challenges is the accurate reconstruction of functions from discrete data. The seminal work of Whittaker, Kotelnikov, and Shannon \cite{butzer1983WKS}, who proved that any band-limited signal can be perfectly reconstructed from its uniformly spaced samples, marked a major milestone in this field and is now known as the WKS sampling theorem. Subsequently, Butzer and Stens \cite{butzer1992} extended this framework to signals that are not necessarily band-limited by replacing the sinc function with more general kernels satisfying suitable conditions, inspiring extensive research on generalized sampling and reconstruction methods (see \cite{k2007,butzer1993, tam}).

An important development in this context is the introduction of exponential sampling, tracing back to the works of Bertero and Pike \cite{ bertero}, Gori \cite{ gori}, Ostrowsky et al. \cite{ostrowsky}, who derived a reconstruction formula for band-limited Mellin functions, later formalized as the exponential sampling formula. The Mellin analysis framework, initially developed by Mamedov \cite{mamedov} and refined by Butzer and Jansche \cite{butzer1997, butzer1998}, provides a natural foundation for addressing sampling and approximation problems involving exponentially spaced data. This framework has been found widely applied in areas such as radio astronomy, optical physics, and photon correlation spectroscopy (see \cite{ bertero, gori, ostrowsky}).

To approximate functions that are not Mellin band-limited, Bardaro and Mantellini \cite{bardaro2017generalized} extended the exponential sampling formula using generalized kernels satisfying suitable smoothness and integrability conditions (see also, \cite{comboexp, bardaro2019, diskant}). 
Among the different classes of exponential sampling operators, two families have attracted particular attention due to their strong approximation properties: the Kantorovich $(K_{\Phi,n})$ \cite{Ang2020} and Durrmeyer $(D_{\Phi,\Psi,n})$ \cite{bardaro2021durrmeyer} operators, defined respectively as
\begin{align*}
   (K_{\Phi,n}f)(x) &= n \sum_{k=-\infty}^{\infty} 
  \Phi(e^{-k}x^{n}) \int_{e^{k/n}}^{e^{(k+1)/n}} f(u) \frac{du}{u},\\ (D_{\Phi,\Psi,n}f)(x)&= n \sum_{k=-\infty}^{\infty} \Phi(e^{-k}x^{n})
  \int_{0}^{\infty} \Psi(e^{-k}u^{n}) f(u) \frac{du}{u}.
\end{align*} where $x \in \mathbb{R}_{+},  n \in \mathbb{N}$.
The Kantorovich and Durrmeyer modifications extend the approximation process to broader functional settings, including Orlicz and $L^{p}$-spaces, ensuring stability under integration and reducing sampling distortions such as time jitter and round-off errors. Moreover, the Durrmeyer framework serves as a unifying structure for several operator classes: choosing $\Psi = \chi_{[0,1]}$ recovers the Kantorovich operator, while setting $\Psi = \delta$ (the Dirac distribution) yields the generalized form. 

Recent studies \cite{cai2024convergence, costarelli2023convergence,costarelli2023quantitative} have demonstrated the effectiveness of the Orlicz space framework for analyzing convergence and approximation properties of sampling operators, including Kantorovich and Durrmeyer-type constructions (see also \cite{Acar2025,Acar2023a,aral2022,bajpeyi2022approximation,coroianu2024approximation}. The framework has also proven instrumental for studying nonlinear operators, particularly max-product and max-min types, due to its flexibility and robustness in handling nonstandard growth behaviors. 

The theory of max-product operators was initiated by Coroianu and Gal \cite{CG2010,CG2011,CG2012}, who introduced generalized, sinc-type and truncated max-product sampling operators and establishing their convergence and saturation properties. These results were later unified and extended into a comprehensive framework for nonlinear approximation \cite{bede2016max}. Exponential sampling based Kantorovich max-product neural network operators were subsequently developed \cite{bajpeyi2024exponential}, extending the approach to neural approximation theory. Convergence in weighted spaces was further analyzed \cite{pradhan2024, pradhan2025}, and within the Orlicz framework, modular convergence for max-product Kantorovich operators was established \cite{boccali2024max}. Univariate and multivariate neural network based max-product approximations were also studied, demonstrating modular convergence and enhancing their applicability in nonlinear settings \cite{costarelli2018, costarelli2019}.

In parallel, the development of max-min sampling operators began with a general approximation theory \cite{gokcer2020}, later extended to Kantorovich-type max-min operators with improved convergence properties \cite{gokcer2022a} and subsequently generalized to neural network operators \cite{aslan2025}. Within the Orlicz framework, the approximation properties of generalized and Kantrovich-type max-min sampling operators have been investigated, particularly in image processing applications \cite{vayeda2025}, with uniform and modular convergence of max-min exponential sampling neural network operators established \cite{pradhan2025a}.

Motivated by these advances, the present study develops and analyzes max-product and max-min Durrmeyer-type exponential sampling operators within the Orlicz space framework. We establish their pointwise and uniform convergence, derive error estimates, and compare their approximation efficiency across various Mellin-type kernel configurations.

Section \ref{section2} presents the necessary mathematical background on Orlicz spaces, exponential kernels, and Mellin analysis, together with supporting lemmas. 
Section \ref{sec3} develops the max-product Durrmeyer-type exponential sampling operators and establishes their pointwise, uniform, and modular convergence. 
In Section \ref{sec4} provides a parallel treatment for the max-min counterparts, highlighting structural and comparative convergence features.
Section \ref{sec5} presents numerical results and graphical demonstrations that validate the theoretical findings.
Finally, Section \ref{sec:conclusion} concludes the article with a summary of contributions and prospective research directions in nonlinear exponential approximation within Orlicz spaces.

\section{Preliminaries}\label{section2} 
Let $\mathbb{R}_{+}$ denote the set of positive real numbers. For $1 \leq p < \infty$, let $\mathscr{L}^{p}(\mathbb{R}_{+})$ denote the space of Lebesgue measurable, $p$-integrable functions on $\mathbb{R}_{+}$ equipped with the usual $p$-norm, and let $\mathscr{L}^{\infty}(\mathbb{R}_{+})$ denote the space of bounded functions endowed with the sup-norm. 

 We begin by recalling some basic notions of Mellin analysis, which provides a natural framework for studying functions defined on the positive real axis and is particularly suited to exponential sampling.

For $c \in \mathbb{R}$, define the space
\begin{align*}
    X_{c} := \left\{ h : \mathbb{R}_{+} \to \mathbb{R} \big| h(w)w^{c-1} \in \mathscr{L}^{1}(\mathbb{R}_{+}) \right\},
\end{align*}
equipped with the norm
\begin{align*}
\|h\|_{X_{c}} = \int_{0}^{\infty} |h(w)|w^{c-1}dw.
\end{align*}
For $h \in X_{c}$, the Mellin transform is defined by
\begin{align*}
M[h](s) = \int_{0}^{\infty} h(w)w^{s-1}dw, 
\quad (s = c+it, t \in \mathbb{R}).
\end{align*}
\noindent
For differentiable functions, the Mellin differential operator $\Theta_{c}$ is given by
\begin{align*}(\Theta_{c}h)(w) := w h'(w) + c h(w), \quad w \in \mathbb{R}_{+}\end{align*}
and higher-order derivatives are defined recursively by
\begin{align*}\Theta_{c}^{r} := \Theta_{c}\big(\Theta_{c}^{r-1}\big), \quad (r \geq 2)\quad \text{with}\quad \Theta_{c}^{0} h := h.\end{align*}
For detailed discussion of the fundamental properties of Mellin transform (see \cite{butzer1997}).
 Throughout this work, we consider $\mathscr{I} $ to be any compact subset of $\mathbb{R}_{+}$.
\begin{definition}\label{def1}
    A function $\mathrm{h}:\mathbb{R}_{+} \rightarrow \mathbb{R}$ is said to be $\log$ continuous at a point $w \in\mathbb{R}_{+} $ if for every $\epsilon > 0$, there exists $\varrho > 0 $ such that for all $ v \in \mathbb{R}_{+} $,
    \begin{align*}
    |\log v -\log w| \leq \varrho \quad \Longrightarrow  \quad |h(v) - h(w)| < \epsilon .\end{align*}
\end{definition}
\begin{definition}\label{def1}
    A function $h:\mathbb{R}_{+} \rightarrow \mathbb{R}$ is said to be $\log$-uniformly continuous  if for every $\epsilon > 0$, there exists $\varrho > 0 $ such that for all $w_{1},w_{2} \in \mathbb{R}_{+}$, 
    \begin{align*}
    |\log w_{1} -\log w_{2}| \leq \varrho \quad \Longrightarrow  \quad |h(w_{1}) - h(w_{2})| < \epsilon .
    \end{align*}
\end{definition}
\noindent
We define the function space
 \begin{align*}
 \mathscr{U}_{bl}(\mathscr{I}):= \{ h: \mathscr{I} \to \mathbb{R} \big|  h~ \text{is  a bounded}, ~\log\text{-uniformly continuous function}\},
 \end{align*} 
 equipped with the supremum norm $|h\|_{\infty}:=\sup\limits_{w\in \mathscr{I}} |\mathrm{h}(w)|$.
    
It is evident that log-uniform continuity and uniform continuity are equivalent on any compact subset of $\mathbb{R}_{+}.$
These notions will be useful in our analysis of the max-product and max-min exponential sampling operators.

We now recall some fundamental concepts related to Orlicz spaces as a generalization of the classical Lebesgue spaces.
Define the Haar measure $ \mathfrak{h}$ on $ \mathbb{R}_+ $ by:
\begin{align*}
\mathfrak{h}(\Omega) = \int_\Omega \frac{dt}{t}, \quad \text{for any measurable set } \Omega \subset \mathbb{R}_+.
\end{align*}  
Let $\mathrm{M}(\Omega, \mathfrak{h}) $ denote the space of measurable functions on $\Omega$ with respect to $\mathfrak{h}$.\\

\noindent
Let $\varphi : \mathbb{R}^+_0 \to \mathbb{R}^+_0$ be a function satisfy:
\begin{align*}
(\varphi1)& : \varphi(0) = 0 ~\text{and }~ \varphi(v) > 0 ~ \text{for all }~ v > 0, \\
(\varphi2)&:  \varphi ~\text{is continuous and non-decreasing on }~ \mathbb{R}^+_0, \\
(\varphi3)&: \lim_{v \to +\infty} \varphi(v) = +\infty,\\
(\varphi4)&:   \varphi ~ \text{ is  convex  on }~ \mathbb{R}_0^+. 
\end{align*}
Such a function is called a convex $\varphi$-function.

Let a convex $\varphi$-function $\zeta$ and $h \in \mathrm{M}(\mathbb{R}_{+}, \mathfrak{h})$,  define the modular functional 
\begin{align*}
I_\zeta[h]:= \int_0^\infty \zeta(|h(w)|)  \frac{dw}{w}.
\end{align*}
It is evident that $I_\zeta$ is convex.\\

\noindent
The Orlicz space over $\mathbb{R}_+$ with respect to $\mathfrak{h}$ and $\zeta$ is defined by
\begin{align*}
L^\zeta_\mathfrak{h}(\mathbb{R}_+) := \left\{h\in \mathrm{M}(\mathbb{R}_+, \mathfrak{h}): \exists ~ \ell > 0\quad \text{ such that }\quad I_\zeta[\ell h] < \infty \right\}.
\end{align*}
For a fixed $ c \in \mathbb{R} $, the Mellin-Orlicz space is defined as
\begin{align*}
\mathscr{X}_c^\zeta := \left\{h : \mathbb{R}_+ \to \mathbb{R} \middle| h(\cdot) (\cdot)^c \in L^\zeta_\mathfrak{h}(\mathbb{R}_+) \right\}.
\end{align*}
In this paper, we focus on the case for $ c = 0 $, where
$\mathscr{X}_0^\zeta = L^\zeta_\mathfrak{h}(\mathbb{R}_+).$\\

\noindent
Define the subspace of finite elements as follows:
\begin{align*}
\mathscr{E}^\zeta_\mathfrak{h}(\mathbb{R}_+) := \left\{ h \in \mathrm{M}(\mathbb{R}_+, \mathfrak{h}) : I_\zeta[\ell h] < \infty \quad\text{ for all } \quad \ell > 0 \right\}.
\end{align*}
The Luxemburg norm is given by
\begin{align*}
\|h\|_\zeta := \inf \left\{ \ell > 0 : I_\zeta\left[\frac{h}{\ell}\right] \leq 1 \right\}.
\end{align*}

\noindent
Let $h_m \subset  L^\zeta_\mathfrak{h}(\mathbb{R}_+), \quad(m>0).$ Then
\begin{itemize}
    \item $h_m\rightarrow h$ modularly if there exists $\ell > 0$ such that
    \begin{align*}
    \lim_{m \to +\infty} I_\zeta[\ell(h_m - h)] = 0;
   \end{align*}
    \item $h_m \rightarrow h$ in Luxemburg norm if
    \begin{align*}
    \lim_{m \to +\infty} \|\mathrm{h}_m - \mathrm{h}\|_\zeta = 0,
    \end{align*}
equivalently, $\displaystyle \lim_{m \to +\infty} I_\zeta[\ell(h_m - h)] = 0$ for all  $\ell > 0$. 
\end{itemize}
A convex $\varphi-$function $\zeta$ is said to satisfy the \textbf{$ \Delta_2 $-condition} if 
there exists a constant $M >0$ such that 
\begin{align*}
\zeta(2v) \leq M \zeta(v), \quad \forall~ v \geq 0.
\end{align*}
Luxemburg norm convergence is stronger than modular convergence; however, if $\zeta$ satisfies the $\Delta_2$-condition, the two notions coincide. In that case,
\begin{align*}
L^\zeta_\mathfrak{h}(\mathbb{R}_+) = \mathscr{E}^\zeta_\mathfrak{h}(\mathbb{R}_+).
\end{align*}

\begin{example}\label{ex1}
 Let \( p > 1 \) and define $\zeta(v) = v^p.$ Then $$I_\zeta[h] = \int\limits_0^\infty |h(w)|^p  \frac{dw}{w},$$
and \( L^\zeta_\mu(\mathbb{R}_+)= L^p_\mu(\mathbb{R}_+)\).\\
\noindent
The  corresponding Mellin-Orlicz space is
\begin{align*}
X_c^p := \left\{ h \in L^p_\mu(\mathbb{R}_+): \int_0^\infty |h(w)|^p w^{cp - 1} dw < \infty \right\}.
\end{align*}
Here, $\zeta(v)$ satisfies the $ \Delta_2 $-condition, ensuring the equivalence of modular and norm convergence.   
\end{example}

\begin{example}\label{ex2}
 Let $\alpha >0 $ and define $\zeta_{\alpha}(v) = e^{v^{\alpha}} -1$. Then  
\begin{align*}
I_{\zeta_{\alpha}}[h] = \int\limits_{a}^{b} \left[ e^{\left( |h(w)|^{\alpha} \right)} - 1 \right] \frac{dw}{w}, \quad h \in \mathrm{M}(\mathbb{R}_{+},\mathfrak{h}).
\end{align*} 
The associated Orlicz space over the Haar measure $\mu$ is given by
\begin{align*}
L_{\mu}^{\zeta_{\alpha}}(\mathbb{R}_{+}) := \left\{h \in \mathrm{M}(\mathbb{R}_{+}, \mu)  \big|  \int\limits_{a}^{b} \left[ e^{|h(w)|^{\alpha}} - 1 \right] \frac{dw}{w} < \infty \right\}.
\end{align*} 
The associated Mellin-Orlicz space is
\begin{align*}
X_{c}^{\zeta_{\alpha}} := \left\{ h  \in L_{\mu}^{\zeta_{\alpha}}(\mathbb{R}_{+}) :  \int\limits_{a}^{b} \left[ e^{|h(w)|^{\alpha}} - 1 \right]  w^{c-1} dw < \infty \right\}.
\end{align*} 
\noindent
Since $\zeta_{\alpha}$ does not satisfy the $\Delta_{2}$-condition, modular and Luxemburg  convergence are not equivalent.   
\end{example}

\begin{example}
Let $\varphi_{\alpha,\beta} : \mathbb{R}^+_0 \to \mathbb{R}^+_0$ be defined by
    \begin{align*}
\varphi_{\alpha,\beta}(v) = v^\alpha \log^\beta(v + 1), \quad \alpha \geq 1,  \beta > 0.
\end{align*}
Then
\begin{align*}
I_{\varphi_{\alpha,\beta}}[h]=\int_{0}^{\infty}|h(w)|^{\alpha}\log^{\beta} \big(1+|h(w)|\big) \frac{dw}{w}.
\end{align*}
The associated Mellin-Orlicz space is
\begin{align*}
  X_{c}^{\varphi_{\alpha,\beta}}
=\left\{h\in L_{\mu}^{\varphi_{\alpha,\beta}}(\mathbb{R}_{+}) :
\int_{0}^{\infty}|h(w)|^{\alpha}\log^{\beta} \big(1+|h(w)|\big) w^{c-1}dw<\infty\right\}.
\end{align*}
Moreover $\varphi_{\alpha,\beta}$ satisfies the global $\Delta_{2}$-condition; in fact
\begin{align*}
\varphi_{\alpha,\beta}(2v)\le 2^{\alpha+\beta} \varphi_{\alpha,\beta}(v)\quad(v\ge0),
\end{align*}
so modular and norm convergence in the associated Mellin-Orlicz space are equivalent.
\end{example}

\noindent
For any index set $ \mathfrak{J} \subseteq \mathrm{Z},$ define 
\begin{equation}
\bigvee_{j \in \mathfrak{J}} v_j :=
\begin{cases}
\sup\limits_{j \in \mathfrak{J}}\,\{ v_j\}, & \text{if $|\mathfrak{J}|=\infty$ }, \\
\max\limits_{j \in \mathfrak{J}}\,\{ v_j\}, & \text{if $|\mathfrak{J}|< \infty$ },
\end{cases}
\quad
\text{and}
\quad
\bigwedge_{j \in \mathfrak{J}} w_j :=
\begin{cases}
\inf\limits_{j \in \mathfrak{J}}\,\{ w_j\}, & \text{if $|\mathfrak{J}|=\infty$ }, \\
\min\limits_{j \in \mathfrak{J}}\,\{ w_j\},  & \text{if $|\mathfrak{J}|< \infty$ }.
\end{cases}
\end{equation}
\noindent
Let $\Phi : \mathbb{R}_{+} \to [0,+\infty)$ be a bounded and $L^{1}$-integrable function on $\mathbb{R}_{+}$, satisfying the following conditions:
\begin{description}
\item[$(\Phi.{1}):$] There exists $ r > 0 $ such that the discrete absolute moment of order $r$,
    \begin{align*}
      \mathfrak{M}_{r}(\Phi) = \sup_{w \in \mathbb{R}_{+}} \bigvee_{k \in \mathbb{Z}}|\Phi(e^{-k}w)||k-\log w|^{r}\quad \text{is finite for} \quad r = 2 ;
    \end{align*}
\item[$(\Phi.{2}):$] The function $\Phi$ satisfies
    \begin{align*}
    \inf_{w \in [1,e]} \Phi(w) =: \vartheta_{w}.
    \end{align*}
\end{description}
Let $\widetilde{\Phi}$ denote the class of functions $\Phi$ that fulfill the conditions $(\Phi.1)$ and $(\Phi.2)$.
\noindent
We state the lemmas \ref{lma1}, \ref{lma2}, and \ref{lm3} without proof, as their arguments are identical to those presented in \cite{Ang1,pradhan2025a}.
\begin{lemma}\label{lma1}
 Let $\Phi$ be a bounded function satisfying condition $(\Phi.{1})$, and let $\mu > 0$. Then
 \begin{align*}
 \mathfrak{M}_{\nu}(\Phi) < \infty, \quad \text{ for every }\quad 0 \leq \nu \leq \mu.
 \end{align*}
\end{lemma}

\begin{lemma}\label{lma2}
  Let $\Phi\in\widetilde{\Phi}$ satisfy $ \mathfrak{M}_{\nu}(\Phi) < \infty$ for all $\nu > 0$. Then, for any $\varrho >0$,  
  \begin{equation*}
  \bigvee_{\substack{k\in \mathfrak{I}_{n}\\|k-\log w| >  n \varrho}}|\Phi(e^{-k}w)| = \mathcal{O}(n^{-\nu}), \quad \text{as}\quad n \rightarrow \infty,
  \end{equation*}  
   uniformly with respect to $ w \in \mathbb{R_{+}}$.
\end{lemma}

\begin{lemma}\label{lm3}
    Let $\Phi : \mathbb{R_{+}}\rightarrow \mathbb{R}_{+}$ be a kernel satisfying condition $(\Phi.{2})$. Then 
    \begin{enumerate}
        \item For any $w \in \mathscr{I}$, we have
        \begin{equation*}
            \bigvee_{k\in \mathrm{J}_{n} } |\Phi(e^{-k}w^{n})| \geq \vartheta_{w},
        \end{equation*} 
        where $\mathrm{J}_{n} = \{ k \in \mathbb{Z} : k = \lceil{n\log a}\rceil,\cdots,\lfloor{n\log b}\rfloor \}.$
        \item For any $w \in \mathbb{R_{+}}$, we have  
        \begin{equation*}
        \bigvee_{k\in \mathrm{Z}} |\Phi(e^{-k}w^{n})| \geq \vartheta_{w}. 
        \end{equation*}
    \end{enumerate}  
\end{lemma}

\begin{lemma}\label{lma_phi}\cite{costarelli2018}
Let $\varphi$ be a convex $\varphi$-function. 
Then, for any index set $\mathfrak{J} \subseteq \mathbb{Z}$ and nonnegative sequence $\{A_k\}_{k \in \mathfrak{J}}$, the following inequality holds:
\begin{align*}
\varphi\!\left( \bigvee_{k \in \mathfrak{J}} A_k \right)\leq \bigvee_{k \in \mathfrak{J}} \varphi(2A_k).
\end{align*} 
\end{lemma}

\begin{definition}
Let $\psi : \mathbb{R}_{+} \to [0,+\infty)$ be a bounded and $L^{1}$-integrable function on $\mathbb{R}_{+}$, satisfying the following conditions:
\begin{description}
    \item[$(\psi.1)$] $ \int\limits_{1}^{e} \psi(w)\frac{dw}{w} =: K > 0.$
    \item[$(\psi.2)$] The discrete absolute moment of order $0$ is finite, i.e.,
    \begin{align*}
    \mathscr{M}_{0}(\psi) := \sup_{v \in \mathbb{R}} \sum_{k \in \mathbb{Z}} \psi(ve^{-k}) < +\infty.
    \end{align*} 
    \item[$(\psi.3)$] The continuous absolute moment of order $0$ is finite, i.e.,
    \begin{align*}
    \displaystyle M_{0}(\psi)  := \int_{a}^{b} |\psi(w)|\frac{dw}{w} < \|\psi \|_{1} < \infty.
    \end{align*} 
\end{description}
\end{definition}

\begin{definition}
    For any $r \in \mathbb{Z}_{+}\cup{\{0\}}$, the continuous absolute moment of order $r$ of $\psi$ is defined by
   \begin{align*}
   M_{r}(\psi):= \int\limits_{a}^{b} |\psi(w)|  |\log w|^{r} \frac{dw}{w}.
   \end{align*} 
\end{definition}
\begin{remark}
    Let $r_{1}, r_{2} \in \mathbb{Z}_{+}$ with $r_{1}< r_{2}.$ Then $M_{r_1}(\psi) < \infty$ implies $M_{r_2}(\psi) < \infty.$
\end{remark}
\noindent
Let $\widetilde{\psi}$ denote the class of functions $\psi$ that satisfy the conditions $(\psi.1)$ and $(\psi.2)$.

\begin{lemma}\label{mainlma}
Let $ w \in \mathscr{I}$ and $n \in \mathbb{N}$  such that $\frac{b}{a} > e^{\frac{1}{n}}$. Then, for $\Phi\in\widetilde{\Phi}$, $\psi\in\widetilde{\psi}$, the following inequality holds:
    \begin{align*}
    \bigvee\limits_{k\in \mathfrak{I}_{n}} \Phi(e^{-k}w^{n}) n \int\limits_{a}^{b} \psi(e^{-k}v^{n}) \frac{dv}{v} \geq K\vartheta_{w}. 
    \end{align*} 
\end{lemma}
\begin{proof} We have
\begin{align*}
  \bigvee\limits_{k\in \mathfrak{I}_{n}}  \Phi(e^{-k}w^{n}) n \int\limits_{a}^{b} \psi(e^{-k}v^{n}) \frac{dv}{v} = & \bigvee\limits_{k\in \mathfrak{I}_{n}} \left[ n \int\limits_{a}^{b} \psi(e^{-k}v^{n}) \frac{dv}{v}\right]\Phi(e^{-k}w^{n}).
    \end{align*}
Setting $t = e^{-k}v^{n}$ and using the non-negativity of $\psi$, together with the inclusion 
\begin{align*}
[1,e]\subset [e^{-k}, e^{-k}\left(\frac{b}{a}\right)^{n}]\quad\text{ where } \quad k=0,1,\cdots,n(\log{b}-\log{a})-1, 
\end{align*}
we obtain
\begin{align*}
  \bigvee\limits_{k\in \mathfrak{I}_{n}}  \Phi(e^{-k}w^{n}) n \int\limits_{a}^{b} \psi(e^{-k}v^{n}) \frac{dv}{v}& =  \bigvee\limits_{k\in \mathfrak{I}_{n}} \left[ \int\limits_{e^{-k}a^{n}}^{e^{-k}b^{n}} \psi(t) \frac{dt}{t}\right]\Phi(e^{-k}w^{n})\\
 & \geq \bigvee\limits_{k\in \mathfrak{I}_{n}} \left[ \int\limits_{1}^{e} \psi(t) \frac{dt}{t}\right]\Phi(e^{-k}w^{n}).
    \end{align*}

    Applying condition ($\psi .1$) and Lemma \ref{lm3}, we have
    
\begin{align*}
  \bigvee\limits_{k\in \mathfrak{I}_{n}}  \Phi(e^{-k}w^{n}) n \int\limits_{a}^{b} \psi(e^{-k}v^{n}) \frac{dv}{v} 
      & \geq \bigvee\limits_{k\in \mathfrak{I}_{n}} K \Phi(e^{-k}w^{n})\\
      & = K\bigvee\limits_{k\in \mathfrak{I}_{n}} \Phi(e^{-k}w^{n})\\
   &  \geq K\vartheta_{w}.
    \end{align*}
\end{proof}

\begin{lemma}\label{lm5} (see \cite{aslan2025})                                                        
If $\bigvee\limits_{m \in \mathbb{Z}}d_{m} < \infty$ or $\bigvee\limits_{m \in \mathbb{Z}}e_{m} < \infty $, then the following inequality holds: 
\begin{align*}
\bigvee\limits_{m \in \mathbb{Z}}d_{m}  - \bigvee\limits_{m \in \mathbb{Z}}e_{m} \leq \bigvee\limits_{m \in \mathbb{Z}}\left|d_{m}-e_{m} \right|.
\end{align*} 
\end{lemma}
\begin{lemma}\label{lm6} (see \cite{bede2008pseudo})
    For any $\mu,\nu, s \in [0,1]$, we have
    \begin{align*}
    \left|\mu \land \nu - \mu \land s\right|\leq \mu \land \left| \nu-s\right|, 
    \end{align*}
    % where $\mu \land \nu$ denotes $\min\{\mu,\nu\}.$
\end{lemma}
\begin{lemma}\label{lm7}(see \cite{bede2008pseudo})
    For any $\mu,s,\nu \geq 0$, the following inequality holds:
    \begin{align*}
    \mu\land \nu + s \land \nu \geq  (\mu+s)\land \nu .
    \end{align*}
\end{lemma}
\begin{lemma}\label{lm8}(see\cite{vayeda2025})
   Let $\mathfrak{J}\subset\mathbb{Z}_+ $ be a finite index set, and $\lambda > 0$. Then, for sequences $ \{d_m\}_{m \in \mathfrak{J}} $ and $ \{e_m\}_{m \in \mathfrak{J}} $ taking values in $[0,1] $, we have
\begin{align*}
\lambda \bigvee_{m \in \mathfrak{J}} (d_m \wedge e_m) = \bigvee_{m \in \mathfrak{J}} (\lambda d_m \wedge \lambda e_m).
\end{align*}
\end{lemma}

\section{Max-product Durrmeyer type exponential sampling operators} \label{sec3}
In this section, we derive the pointwise and uniform convergence properties of the max-product Durrmeyer-type exponential sampling operators in the space $\mathscr{U}_{bl}(\mathscr{I})$, and further establish their modular convergence in the  Orlicz type spaces.
\begin{definition}\label{Maxdef1}
    Let $h: \mathscr{I} \rightarrow \mathbb{R}$ be a bounded and $L^1$-integrable function on $\mathscr{I}$. Then the  max-product Durrmeyer-type exponential sampling operators associated with $h$, with respect to the kernels $\Phi$ and $\psi$, is defined by
    \begin{align*}
    \mathscr{D}_{n,\Phi,\psi}^{M}(h)(w) := \frac{\bigvee\limits_{k\in \mathfrak{I}_{n}} \Phi(e^{-k}w^{n}) n \int\limits_{a}^{b} \psi(e^{-k}v^{n}) h(v)\frac{dv}{v}} {\bigvee\limits_{ k\in \mathfrak{I}_{n}} \Phi(e^{-k}w^{n}) n \int\limits_{a}^{b} \psi(e^{-k}v^{n}) \frac{dv}{v}},
    \end{align*}
    where $\mathrm{J}_{n} = \{ k \in \mathbb{Z} : k = \lceil{n\log a}\rceil,\cdots,\lfloor{n\log b}\rfloor\}.$
\end{definition}

\begin{lemma}\cite{pradhan2025b}\label{Maxlma1}
    Let $h,g $ be two locally integrable functions on $\mathscr{I}$. Then the following properties hold for the max-product exponential sampling operators $\mathscr{D}_{n,\Phi,\psi}^{M}(h) :$
    \begin{enumerate}
        \item If $h(w) \leq g(w),$ then $\mathscr{D}_{n,\Phi, \psi}^{M}(h)(w) \leq \mathscr{D}_{n,\Phi,\psi}^{M}(g)(w) $.
        \item $\mathscr{D}_{n,\Phi,\psi}^{M}(h+g)(w) \leq \mathscr{D}_{n,\Phi,\psi}^{M}(h)(w)+ \mathscr{D}_{n,\Phi,\psi}^{M}(g)(w)$.
        \item For all $w\in \mathscr{I}$, $|\mathscr{D}_{n,\Phi,\psi}^{M}(h)(w)- \mathscr{D}_{n,\Phi,\psi}^{M}(g)(w)| \leq \mathscr{D}_{n,\Phi,\psi}^{M}(|h-g|)(w)$.
        \item For every $\lambda > 0$, we have $\mathscr{D}_{n,\Phi,\psi}^{M}(\lambda h)(w  ) = \lambda  \mathscr{D}_{n,\Phi,\psi}^{M}(h)(w)$.
    \end{enumerate}
\end{lemma}
%\begin{proof}
%    Properties (1), (2), and (4) follow directly from the definition of the operator $\mathscr{D}_{n,\Phi,\psi}^{M}$.
%We now prove (3).
%
%Observe that for all $u \in \mathscr{I}$, 
%\begin{align*}
%    f(u) \leq |f(u) - g(u)| + g(u)\quad\text{and}\quad g(u) \leq |g(u) - f(u)| + f(u). 
%\end{align*}
%Applying properties (1) and (2), we obtain
%\begin{equation*}
%\mathscr{D}_{n,\Phi,\psi}^{M}(h)(u) \leq \mathscr{D}_{n,\Phi,\psi}^{M}(f-g)(u)+ \mathscr{D}_{n,\Phi,\psi}^{M}(g)(u),
%\end{equation*} and similarly,
%\begin{equation*} 
%\mathscr{D}_{n,\Phi,\psi}^{M}(g)(u) \leq \mathscr{D}_{n,\Phi,\psi}^{M}(|g-f|)(u)+ \mathscr{D}_{n,\Phi,\psi}^{M}(h)(u).
%\end{equation*}
%    Combining these two inequalities yields
%    \begin{equation*}
%    |\mathscr{D}_{n,\Phi,\psi}^{M}(h)(u)-\mathscr{D}_{n,\Phi,\psi}^{M}(g)(u)| \leq \mathscr{D}_{n,\Phi,\psi}^{M}(|f-g|)(u), \quad \forall~ u\in \mathscr{I}.
%    \end{equation*}
%\end{proof}

\begin{theorem}\label{Maxthm1}
    Let $h:\mathscr{I} \to [0,1]\,$ be a bounded and $L^{1}$-integrable function. Then, for every point of logarithmic continuity $w\in \mathscr{I} \subseteq \mathbb{R}_{+}$, we have
    \begin{align*}
    \lim\limits_{n\to \infty} \mathscr{D}_{n,\Phi, \psi}^{M}(h) (w) = h(w). 
    \end{align*}
Moreover, if $h \in \mathscr{U}_{bl}(\mathscr{I})$, then the convergence is uniform, i.e,
\begin{align*}
\lim\limits_{n\to \infty} \left\|\mathscr{D}_{n,\Phi, \psi}^{M}(h) - h\right\|_{\infty} = 0,
\end{align*}
where $\| .\|_{\infty}$ denotes the supremum norm.
\end{theorem}
\begin{proof}
    From Definition~\ref{Maxdef1}, we can write
    \begin{align*}
       \left|\mathscr{D}_{n,\Phi, \psi}^{M}(h) (w) - h(w) \right| & = \left| \frac{\bigvee\limits_{k\in \mathfrak{I}_{n}} \Phi(e^{-k}w^{n}) n \int\limits_{a}^{b} \psi(e^{-k}v^{n}) h(v)\frac{dv}{v}} {\bigvee\limits_{ k\in \mathfrak{I}_{n}} \Phi(e^{-k}w^{n}) n \int\limits_{a}^{b} \psi(e^{-k}v^{n}) \frac{dv}{v}} - h(w)\right|\\
       & = \left| \frac{\bigvee\limits_{k\in \mathfrak{I}_{n}} \Phi(e^{-k}w^{n}) n \int\limits_{a}^{b} \psi(e^{-k}v^{n}) h(v)\frac{dv}{v}} {\bigvee\limits_{ k\in \mathfrak{I}_{n}} \Phi(e^{-k}w^{n}) n \int\limits_{a}^{b} \psi(e^{-k}v^{n}) \frac{dv}{v}}\right. \\& \left.\qquad - \frac{\bigvee\limits_{k\in \mathfrak{I}_{n}} \Phi(e^{-k}w^{n}) n \int\limits_{a}^{b} \psi(e^{-k}v^{n}) h(w)\frac{dv}{v}} {\bigvee\limits_{ k\in \mathfrak{I}_{n}} \Phi(e^{-k}w^{n}) n \int\limits_{a}^{b} \psi(e^{-k}v^{n}) \frac{dv}{v}}\right|.
    \end{align*}
    \vspace{1em}
By applying Lemmas \ref{mainlma},  \ref{lm5} and property (3) of Lemma \ref{Maxlma1}, we obtain
\begin{align*}
 \left|\mathscr{D}_{n,\Phi, \psi}^{M}(h) (w) - h(w) \right| \leq \frac{1}{K\vartheta_{w}} \bigvee\limits_{k\in \mathfrak{I}_{n}} \Phi(e^{-k}w^{n}) n \int\limits_{a}^{b} |\psi(e^{-k}v^{n})| |h(v) - h(w)|\frac{dv}{v}.
    \end{align*}
    Let $\varepsilon >0$ be arbitrarily. Since $h$ is logarithmic continuity at $w\in \mathscr{I}$,  there exists $\tau >0$ such that 
   \begin{align*} 
   | h(v)-h(w)| < \varepsilon \quad\text{whenever}\quad |\log{v}-\log{w}|< \tau. 
    \end{align*} 
    Define the index sets
    \begin{align*}
    N_{1} = \{k \in \mathscr{I}_{n}: \left| \frac{k}{n}-\log{w}\right| \leq \frac{\tau}{2} \} , \quad N_{2} = \{k \in \mathscr{I}_{n} : \left| \frac{k}{n}-\log{w}\right| > \frac{\tau}{2} \} .
    \end{align*}
     Then we can decompose
     \begin{align*}
       \left|\mathscr{D}_{n,\Phi, \psi}^{M}(h) (w) - h(w) \right| \leq \frac{1}{K\vartheta_{w}} \max\{\mathrm{E}_{1}, \mathrm{E}_{2} \},
       \end{align*}
        where
       \begin{align*}
       \mathrm{E}_{1}&=\bigvee\limits_{k\in N_{1}} \Phi(e^{-k}w^{n}) n \int\limits_{a}^{b} |\psi(e^{-k}v^{n})| |h(v) - h(w)|\frac{dv}{v},\\ \mathrm{E}_{2}&= \bigvee\limits_{k\in N_{2}} \Phi(e^{-k}w^{n}) n \int\limits_{a}^{b} |\psi(e^{-k}v^{n})| |h(v) - h(w)|\frac{dv}{v}.
     \end{align*}
     We further decompose $E_{1}$ as
     \begin{align*}
    \mathrm{E}_{1}& = \bigvee\limits_{k\in N_{1}} \Phi(e^{-k}w^{n}) n \left\{\int\limits_{|\log{v} - \frac{k}{n} | < \frac{\tau}{2}} + \int\limits_{|\log{v} - \frac{k}{n} | \geq \frac{\tau}{2}}  \right\}|\psi(e^{-k}v^{n})| |h(v) - h(w)|\frac{dv}{v}\\
    & =  \mathrm{E}_{1.1} + \mathrm{E}_{1.2}.
     \end{align*}
For $k \in \mathbb{N}_{1}$ and $v \in \mathscr{I}$ satisfying
    $\left|\log v - \tfrac{k}{n}\right| < \tfrac{\tau}{2}$, we have
    \begin{align*}
    \left|\log v - \log w \right| \leq \left|\log v - \tfrac{k}{n}\right| + \left|\tfrac{k}{n} - \log w \right|
    < \tfrac{\tau}{2} + \tfrac{\tau}{2}= \tau .
     \end{align*}
    Hence, by the logarithmic continuity of $h$ at $w$, and using the substitution $t = e^{-k}v^{n}$, we obtain
    \begin{align*}
      \mathrm{E}_{1.1} &<  \varepsilon \bigvee\limits_{k\in N_{1}} \Phi(e^{-k}w^{n}) n \int\limits_{|\log{v} - \frac{k}{n} | < \frac{\tau}{2}} |\psi(e^{-k}v^{n})| \frac{dv}{v}\\
      & <  \varepsilon \mathfrak{M}_{0}(\Phi) \int\limits_{\mathbb{R}_{+}} |\psi(t)| \frac{dt}{t} = \varepsilon \mathfrak{M}_{0}(\Phi)  \|\psi\|_{1}.
    \end{align*}
     Since $h$ is bounded and $\psi \in L^{1}(\mathbb{R}_{+})$, for sufficiently large $n$, we also have
     \begin{align*}
       \mathrm{E}_{1.2} < 2 \|h\|_{\infty} \bigvee\limits_{k\in N_{1}} \Phi(e^{-k}w^{n}) n \int\limits_{|\log{v} - \frac{k}{n} | \geq \frac{\tau}{2}} |\psi(e^{-k}v^{n})| \frac{dv}{v}
       <   2 \|h\|_{\infty} \mathfrak{M}_{0}(\Phi)  \varepsilon.
    \end{align*}                   
 Finally, by Lemma \ref{lma2} and the boundedness of $h$, we estimate
    \begin{align*}
   \mathrm{E}_{2} &  = \bigvee\limits_{k\in N_{2}} \Phi(e^{-k}w^{n}) n \int\limits_{a}^{b} |\psi(e^{-k}v^{n})| |h(v) - h(w)|\frac{dv}{v}\\
   & \leq  2\|h\|_{\infty} \|\psi\|_{1} \bigvee\limits_{k\in N_{2}} \Phi(e^{-k}w^{n}) <  2\|h\|_{\infty} \|\psi\|_{1} \varepsilon.
    \end{align*}
    Combining these estimates gives
  \begin{align*}
       \left|\mathscr{D}_{n,\Phi, \psi}^{M}(h) (w) - h(w) \right| \leq  C\varepsilon,
    \end{align*}
  for some constant $C > 0$ independent of $n$.
Hence,
\begin{align*}
    \lim\limits_{n\to \infty} \mathscr{D}_{n,\Phi, \psi}^{M}(h) (w) = h(w). 
    \end{align*}
    Uniform convergence follows in a similar way when $h \in \mathscr{U}_{bl}(\mathscr{I})$.
\end{proof}

%---------------------------------------------------------------------------------
\begin{theorem}\label{Maxthm2}
    Let $h \in L_{\mathfrak{h}}^{\zeta}(\mathscr{I})$ be fixed, and let $\psi$ be a kernel satisfying $\mathscr{M}_{0}(\psi)<\infty$. Then there exists $\lambda>0$ such that  
    \begin{align*}
        I_{\zeta}[\lambda \mathscr{D}_{n,\Phi, \psi}^{M}(h)] \leq \frac{\mathscr{M}_{0}(\psi)\|\Phi \|_{1}}{\mathfrak{M}_{0}(\Phi)\|\psi\|_{1}} I_{\zeta}\left[\frac{\lambda \mathfrak{M}_{0}(\Phi) \|\psi\|_{1}}{K\vartheta_{w}} h\right].
    \end{align*}
\end{theorem}
\begin{proof}We begin with
\begin{align*}
     I_{\zeta}[\lambda \mathscr{D}_{n,\Phi, \psi}^{M}(h)]& = \int\limits_{a}^{b} \zeta \left( \lambda \left| \mathscr{D}_{n,\Phi, \psi}^{M}(h)(w)\right|\right) \frac{dw}{w}\\
     & = \int\limits_{a}^{b} \zeta \left( \lambda \left| \frac{\bigvee\limits_{k\in \mathfrak{I}_{n}} \Phi(e^{-k}w^{n}) n \int\limits_{a}^{b} \psi(e^{-k}v^{n}) h(v)\frac{dv}{v}} {\bigvee\limits_{ k\in \mathfrak{I}_{n}} \Phi(e^{-k}w^{n}) n \int\limits_{a}^{b} \psi(e^{-k}v^{n}) \frac{dv}{v}}\right|\right) \frac{dw}{w}.
   \end{align*}
Since $\vartheta_{w}$ and $K$ are positive constants, it follows that
  \begin{align*}
     I_{\zeta}[\lambda \mathscr{D}_{n,\Phi, \psi}^{M}(h)] 
     & \leq \int\limits_{a}^{b} \zeta \left( \frac{\lambda}{K\vartheta_{w}} \left| \bigvee\limits_{k\in \mathfrak{I}_{n}} \Phi(e^{-k}w^{n}) n \int\limits_{a}^{b} \psi(e^{-k}v^{n}) h(v)\frac{dv}{v} \right|\right) \frac{dw}{w}\\
    &  = \int\limits_{a}^{b} \zeta \left( \frac{\lambda}{K\vartheta_{w}} \bigvee\limits_{k\in \mathfrak{I}_{n}} |\Phi(e^{-k}w^{n})| n \int\limits_{a}^{b} |\psi(e^{-k}v^{n})| |h(v)|\frac{dv}{v} \right) \frac{dw}{w}.
\end{align*}
   since $\zeta$ is non-decreasing and $\mathscr{I}_{n}$ is finite, we have
   \begin{align}\label{eq1}
   \zeta\left(\bigvee\limits_{k \in \mathscr{I}_{n}} a_{k}\right) = \left(\bigvee\limits_{k \in \mathscr{I}_{n}} \zeta(a_{k})\right).
   \end{align}
   Using \eqref{eq1}, the convexity of $\zeta$, and the bound $|\Phi(e^{-k}w^{n})| \leq \mathfrak{M}_{0}(\Phi)$, for all $k\in \mathscr{I}_{n},$ we get
   \begin{align*}
       I_{\zeta}[\lambda \mathscr{D}_{n,\Phi, \psi}^{M}(h)] & \leq \int\limits_{a}^{b}   \bigvee\limits_{k\in \mathfrak{I}_{n}}\zeta \left(\frac{\lambda}{K\vartheta_{w}}|\Phi(e^{-k}w^{n})| n \int\limits_{a}^{b} |\psi(e^{-k}v^{n})| |h(v)|\frac{dv}{v} \right) \frac{dw}{w}\\
        &\leq \int\limits_{a}^{b}   \bigvee\limits_{k\in \mathfrak{I}_{n}} \frac{|\Phi(e^{-k}w^{n})|}{\mathfrak{M}_{0}(\Phi)}\zeta \left(\frac{\lambda \mathfrak{M}_{0}(\Phi)}{K\vartheta_{w}} n \int\limits_{a}^{b} |\psi(e^{-k}v^{n})| |h(v)|\frac{dv}{v} \right) \frac{dw}{w}.
       %&  \leq\int\limits_{a}^{b}   \sum\limits_{k\in \mathfrak{I}_{n}} \frac{|\Phi(e^{-k}w^{n})|}{\mathfrak{M}_{0}(\Phi)}\zeta \left(\frac{\lambda \mathfrak{M}_{0}(\Phi)}{K\vartheta_{w}} n \int\limits_{a}^{b} |\psi(e^{-k}v^{n})| |h(u)|\frac{dv}{v} \right) \frac{dw}{w}
   \end{align*}
Replacing the maximum by summation and applying Jensen’s inequality yields
\begin{align*}
I_{\zeta}[ &\lambda\mathscr{D}_{n,\Phi, \psi}^{M}(h)]\\
&\leq
\frac{1}{\mathfrak{M}_{0}(\Phi)}\int\limits_{a}^{b} \sum\limits_{k\in \mathfrak{I}_{n}} |\Phi(e^{-k}w^{n})|\zeta \left(\frac{\lambda \mathfrak{M}_{0}(\Phi) \|\psi\|_{1}}{K\vartheta_{w}} n \int\limits_{a}^{b} \frac{|\psi(e^{-k}v^{n})|}{\|\psi\|_{1}} |h(v)|\frac{dv}{v} \right) \frac{dw}{w}.
\end{align*}
 By applying the Jensen's inequality again, along with the Fubini-Tonelli theorem and the change of variable $t=e^{-k}v^{n}$, we obtain
  \begin{align*}
    &  I_{\zeta}[\lambda \mathscr{D}_{n,\Phi, \psi}^{M}(h)]\\ & \leq \frac{1}{\mathfrak{M}_{0}(\Phi)\|\psi\|_{1}}\int\limits_{a}^{b}   \sum\limits_{k\in \mathfrak{I}_{n}} |\Phi(e^{-k}w^{n})| \frac{dw}{w} \int\limits_{e^{-k} a^n}^{e^{-k} b^{n}} |\psi(t)|\zeta \left(\frac{\lambda \mathfrak{M}_{0}(\Phi) \|\psi\|_{1}}{K\vartheta_{w}}  |h((te^{k})^{\frac{1}{n}})|\right)\frac{dt}{t}  \\
      &  \leq\frac{1}{\mathfrak{M}_{0}(\Phi)\|\psi\|_{1}} n\int\limits_{a}^{b} |\Phi(e^{-k}w^{n})| \frac{dw}{w} \int\limits_{a}^{b} \sum\limits_{k\in \mathfrak{I}_{n}}|\psi(e^{-k} v^n)|\zeta \left(\frac{\lambda \mathfrak{M}_{0}(\Phi) \|\psi\|_{1}}{K\vartheta_{w}}  |h(v)|\right)\frac{dv}{v}  \\
       & \leq\frac{\mathscr{M}_{0}(\psi)\|\Phi \|_{1}}{\mathfrak{M}_{0}(\Phi)\|\psi\|_{1}}  \int\limits_{a}^{b}\zeta \left(\frac{\lambda \mathfrak{M}_{0}(\Phi) \|\psi\|_{1}}{K\vartheta_{w}}  |h(u)|\right)\frac{dv}{v}.
  \end{align*}
  Thus,
 \begin{align*}
      I_{\zeta}[\lambda \mathscr{D}_{n,\Phi, \psi}^{M}(h)]\leq \frac{\mathscr{M}_{0}(\psi)\|\Phi \|_{1}}{\mathfrak{M}_{0}(\Phi)\|\psi\|_{1}} I_{\zeta}\left[\frac{\lambda \mathfrak{M}_{0}(\Phi) \|\psi\|_{1}}{K\vartheta_{w}} h\right].
  \end{align*}
\end{proof}
%-----------------------------------------------------------------------------------

\begin{theorem}\label{Maxthm3}
    Let $h \in \mathscr{U}_{bl}(\mathscr{I})$. Then, for each $\lambda > 0$, we have
 \begin{align*}   
\lim_{n \to \infty} I_{\zeta}\left[\lambda \left( \mathscr{D}_{n,\Phi, \psi }^{M} (h) - h \right) \right] = 0.
\end{align*}
\end{theorem}
\begin{proof}
    Let $\varepsilon > 0$ be fixed. Then, for any $\lambda > 0$, using the convexity of $\eta$ and Theorem \ref{Maxthm1}, we obtain
\begin{align*}
I_{\zeta}\left[\lambda \left( \mathscr{D}_{n,\Phi, \psi }^{M}\, (h) - h \right) \right] & = \int\limits_{\mathscr{I}} \zeta \left( \left| \lambda \left( \mathscr{D}_{n,\Phi, \psi }^{M} (h)(w) - h(w)  \right)\right|\right) \frac{dw}{w}\\
 &\leq \int\limits_{\mathscr{I}} \zeta \left( \lambda\left\| \left( \mathscr{D}_{n,\Phi, \psi }^{M} (h) - h  \right)\right\|_{\infty}\right) \frac{dw}{w} \leq \int\limits_{\mathscr{I}} \zeta(\lambda  \varepsilon) \frac{dw}{w}\\
&=  \zeta (\lambda \varepsilon) (\log{b} - \log{a})\\
&\leq   \varepsilon  \zeta(\lambda)\log{\left(\frac{b}{a}\right)}.
\end{align*}
Hence, the result follows by the arbitrariness of $\varepsilon > 0$.
\end{proof}

%----------------------------------------------------------------------------------
\begin{lemma}\cite{vayeda2025}\label{denselma}
   The spaces $\mathscr{U}_{bl}(\mathscr{I})$ is dense in $L^{\zeta}_{\mathfrak{h}}(\mathscr{I})$ with respect to the topology induced by modular convergence.
\end{lemma}
%-------------------------------------------------------------------------------------
\begin{theorem}
    Let $h \in L_{\mathfrak{h}}^{\zeta}(\mathscr{I})$ . Then, for each $\lambda > 0$, we have
\begin{align*}
\lim_{n \to \infty} I_{\zeta}\left[\lambda \left( \mathscr{D}_{n,\Phi, \psi }^{M} (h) - h \right) \right] = 0.
\end{align*}
\end{theorem}
\begin{proof}
    Let $\varepsilon > 0 $ be given and choose $\lambda_{*} > 0$. Then, by Theorem \ref{Maxthm3}, there exists $ n_0 \in \mathbb{N}$ such that
    \begin{equation*}
       I_{\zeta}\left[\lambda_{*} \left( \mathscr{D}_{n,\Phi, \psi }^{M} (h) - h \right) \right] < \varepsilon, \quad \forall~ n > n_{0}.
    \end{equation*}
Now, since $h \in L_{\mathfrak{h}}^{\zeta}(\mathscr{I})$ and $\varepsilon >0 $, by Lemma~\ref{denselma},  there exists $ f \in \mathscr{U}_{bl}(\mathscr{I})$ such that 
    \begin{equation}\label{Maxlastthmeqn}
       I_{\zeta}\left[\lambda_{*} \left( h - f \right) \right] < \left(1+ \frac{\mathscr{M}_{0}(\psi)\|\Phi \|_{1}}{\mathfrak{M}_{0}(\Phi)\|\psi\|_{1}}  \right)^{-1} 2\varepsilon.
    \end{equation}
   Using the convexity of $\zeta$, for each $n > n_{0}$, we have    
    \begin{align*}
    I_{\zeta}&\left[\lambda \left( \mathscr{D}_{n,\Phi, \psi }^{M} (h) - h \right) \right]\\ &\leq  I_{\zeta}\left[\lambda \left( \mathscr{D}_{n,\Phi, \psi }^{M} (h)- \mathscr{D}_{n,\Phi, \psi }^{M} (f) + \mathscr{D}_{n,\Phi, \psi }^{M} (f)-f +f - h \right) \right]\\
       & \leq \frac{1}{3} \left\{I_{\zeta}\left[3\lambda \left( \mathscr{D}_{n,\Phi, \psi }^{M} (h)- \mathscr{D}_{n,\Phi, \psi }^{M} (f)\right)\right] + I_{\zeta}\left[3\lambda \left(  \mathscr{D}_{n,\Phi, \psi }^{M} (f)-f\right)\right]\right.\\& \qquad+ \left. I_{\zeta}\left[3\lambda \left( f-h\right)\right]\right\}.
    \end{align*}
   Fix $\lambda >0 $ such that $3\lambda\left( 1 + \frac{\mathfrak{M}_{0}(\Phi) \|\psi\|_{1}}{K\vartheta_{w}}\right) < \lambda_{*}$. Then, applying Theorem \ref{Maxthm2} and Theorem \ref{Maxthm3} together with (\ref{Maxlastthmeqn}), we get    
    \begin{align*}
       I_{\zeta}\left[\lambda\right.&\left. \left( \mathscr{D}_{n,\Phi, \psi }^{M} (h) - h \right) \right] \leq \frac{1}{3} \left\{I_{\zeta}\left[3\lambda \left( \mathscr{D}_{n,\Phi, \psi }^{M} (h)- \mathscr{D}_{n,\Phi, \psi }^{M} (f)\right)\right]+ I_{\zeta}\left[3\lambda \left( f-h\right)\right] \right. \\&\hspace{3cm}\quad+ \left. I_{\zeta}\left[3\lambda \left(  \mathscr{D}_{n,\Phi, \psi }^{M}(f)-f\right)\right] \right\}\\
      &\quad \leq \frac{1}{3}\left\{\left(1+ \frac{\mathscr{M}_{0}(\psi)\|\Phi \|_{1}}{\mathfrak{M}_{0}(\Phi)\|\psi\|_{1}}  \right)I_{\zeta}\left[\lambda_{*} \left( h - f \right) \right] + I_{\zeta}\left[\lambda_{*} \left(  \mathscr{D}_{n,\Phi, \psi }^{M} (f)-f\right)\right]  \right\}\\
       &\quad < \frac{1}{3}\{2\varepsilon + \varepsilon\} = \varepsilon.
    \end{align*}
\end{proof}
%----------------------------------------------------------------------------------------

\section{Max-min Durrmeyer-type exponential sampling operators}\label{sec4}
In this section, we derive the pointwise and uniform convergence proper-
ties of the max-min Durrmeyer-type exponential sampling operators in
the space $\mathscr{U}_{bl}(\mathscr{I})$, and further establish their modular convergence in the
Orlicz  spaces.

\begin{definition}\label{Maxmin}
     Let $h: \mathscr{I} \rightarrow \mathbb{R}$ be a bounded and $L^1$-integrable function on $\mathscr{I}$. Then the max-min Durrmeyer-type exponential sampling operators associated with $h$ and the kernels $\Phi$ and $\psi$ is defined by
     \begin{align*} 
     \mathscr{D}_{n,\Phi, \psi}^{m}(h) (w) := \bigvee\limits_{k\in \mathfrak{I}_{n}}\left[ n \int\limits_{a}^{b} \psi(e^{-k}v^{n})h(v)\frac{dv}{v}\right] \wedge \frac{\Phi(e^{-k}w^{n})}{\bigvee\limits_{k\in \mathfrak{I}_{n}} \Phi(e^{-k}w^{n})n \int\limits_{a}^{b} \psi(e^{-k}v^{n}) \frac{dv}{v}},
     \end{align*} 
     where $\mathfrak{I}_{n} = \lceil{n\log{a}}\rceil, \cdots, \lfloor{n\log{b}}\rfloor.$
\end{definition}
\begin{definition}\label{def3}
    For a given $w \in \mathscr{I}\subseteq \mathbb{R}_{+}, n \in \mathbb{N}$, and $\tau > 0$, we define 
    \begin{align*}
    \mathscr{B}_{\tau, n} : = \{ k = \lceil{n\log{a}}\rceil,\cdots,\lfloor{n\log{b}}\rfloor  :  \left| \frac{k}{n}- \log{w}\right| \leq \tau \}.
    \end{align*}
\end{definition}
\begin{lemma}\label{Maxminlma}\cite{pradhan2025b}
 Let $h, f : \mathscr{I} \to [0,1]$ be two bounded functions. Then the following properties hold:
     \begin{enumerate}
       % \item[(a)] The operator $\mathscr{D}_{n,\Phi, \psi}^{m}(h) (w)$ is log-continuous on  $\mathscr{I}$.\\[2pt]
       \item[(b)] If $h(w) \leq f(w)$ for all $w \in \mathscr{I}$, then $$ \mathscr{D}_{n,\Phi, \psi}^{m}(h) (w) \leq\mathscr{D}_{n,\Phi, \psi}^{m}(f) (w),\quad \forall~ w \in \mathscr{I}.$$
       \item[(c)] For all sufficiently large $n\in \mathbb{N}$ and for every $w \in \mathbb{R_{+}}$, we have
       \begin{align*} 
       \left|\mathscr{D}_{n,\Phi, \psi}^{m}(h) (w) - \mathscr{D}_{n,\Phi, \psi}^{m}(f) (w)\right| \leq \mathscr{D}_{n,\Phi, \psi}^{m}(|h - f|)(w).
       \end{align*}
       \item[(d)]For all $w \in \mathbb{R}_{+}$,  $$\mathscr{D}_{n,\Phi, \psi}^{m}(h + f) (w)\leq \mathscr{D}_{n,\Phi, \psi}^{m}(h) (w)+ \mathscr{D}_{n,\Phi, \psi}^{m}(f) (w).$$      
     \end{enumerate}
\end{lemma}

\begin{theorem}\label{Maxminthm1}
Let $\mathrm{h}:\mathscr{I} \to [0,1]$ be a bounded and $L^{1}$-integrable function. Then, at any point of log continuity $w \in \mathscr{I} \subseteq \mathbb{R}_{+}$,
\begin{align*} 
 \lim\limits_{n\to \infty} \mathscr{D}_{n,\Phi, \psi}^{m}(h) (w) = h(w).
 \end{align*}
Moreover, if $h \in \mathscr{U}_{bl}(\mathscr{I})$, then the convergence is uniform, i.e.,
\begin{align*}
\lim\limits_{n\to \infty} \left\|\mathscr{D}_{n,\Phi, \psi}^{M}(h) - h\right\|_{\infty} = 0,
\end{align*}
where $\| .\|_{\infty}$ denotes the supremum norm.
\end{theorem}
\begin{proof}
   Let $h:\mathscr{I} \to [0,1]$ be log-continuous at $w\in \mathscr{I}$. By the triangle inequality, we have 
    \begin{align*}
    & \left|\mathscr{D}_{n,\Phi, \psi}^{M}(h) (w) - h(w)\right|\\ & = \left| \bigvee\limits_{k\in \mathfrak{I}_{n}}   \left[ n \int\limits_{a}^{b} \psi(e^{-k}v^{n}) h(v)\frac{dv}{v}\right] \wedge \frac{\Phi(e^{-k}w^{n})}{\bigvee\limits_{k\in \mathfrak{I}_{n}} \Phi(e^{-k}w^{n})n \int\limits_{a}^{b} \psi(e^{-k}v^{n}) \frac{dv}{v}} - h(w)\right|\\
      & \leq \left| \bigvee\limits_{k\in \mathfrak{I}_{n}}   \left[ n \int\limits_{a}^{b} \psi(e^{-k}v^{n}) h(v)\frac{dv}{v}\right] \wedge \frac{\Phi(e^{-k}w^{n})}{\bigvee\limits_{k\in \mathfrak{I}_{n}} \Phi(e^{-k}w^{n}) n \int\limits_{a}^{b} \psi(e^{-k}v^{n}) \frac{dv}{v}} \right. \\ & \left. \quad -\bigvee\limits_{k\in \mathfrak{I}_{n}}   \left[ n \int\limits_{a}^{b} \psi(e^{-k}v^{n}) h(w)\frac{dv}{v}\right] \wedge \frac{\Phi(e^{-k}w^{n})}{\bigvee\limits_{k\in \mathfrak{I}_{n}} \Phi(e^{-k}w^{n})n \int\limits_{a}^{b} \psi(e^{-k}v^{n}) \frac{dv}{v}}\right|\\ &+ \left|\bigvee\limits_{k\in \mathfrak{I}_{n}}   \left[ n \int\limits_{a}^{b} \psi(e^{-k}v^{n}) h(w)\frac{dv}{v}\right] \wedge \frac{\Phi(e^{-k}w^{n})}{\bigvee\limits_{k\in \mathfrak{I}_{n}} \Phi(e^{-k}w^{n}) n \int\limits_{a}^{b} \psi(e^{-k}v^{n}) \frac{dv}{v}} - h(w)\right|\\ 
        & =\mathsf{E}_{1} + \mathsf{E}_{2}.
    \end{align*}
     To estimate $\mathsf{E}_{1}$, using Lemmas \ref{lm5} and \ref{lm6}, we write
   \begin{align*}
       \mathsf{E}_{1}% & = \left| \bigvee\limits_{k\in \mathfrak{I}_{n}}   \left[ n \int\limits_{a}^{b} \psi(e^{-k}v^{n}) h(v)\frac{dv}{v}\right] \wedge \frac{\Phi(e^{-k}w^{n})}{\bigvee\limits_{k\in \mathfrak{I}_{n}} \Phi(e^{-k}w^{n}) n \int\limits_{a}^{b} \psi(e^{-k}v^{n}) \frac{dv}{v}} \right. \\ & \left. \quad -\bigvee\limits_{k\in \mathfrak{I}_{n}}   \left[ n \int\limits_{a}^{b} \psi(e^{-k}v^{n}) h(w)\frac{dv}{v}\right] \wedge \frac{\Phi(e^{-k}w^{n})}{\bigvee\limits_{k\in \mathfrak{I}_{n}} \Phi(e^{-k}w^{n}) n \int\limits_{a}^{b} \psi(e^{-k}v^{n}) \frac{dv}{v}}\right|\\ 
        & \leq \bigvee\limits_{k\in \mathfrak{I}_{n}} \left| \left[n \int\limits_{a}^{b} \psi(e^{-k}v^{n}) (h(v) - h(w))\frac{dv}{v}\right]\right|  \wedge \frac{\Phi(e^{-k}w^{n})}{\bigvee\limits_{k\in \mathfrak{I}_{n}} \Phi(e^{-k}w^{n})\, n \int\limits_{a}^{b} \psi(e^{-k}v^{n}) \frac{dv}{v}}\\
       &  \leq \bigvee\limits_{k\in \mathfrak{I}_{n}}  \left[n \int\limits_{a}^{b} \psi(e^{-k}v^{n}) \left|(h(v) - h(w))\right|\frac{dv}{v}\right]  \wedge \frac{\Phi(e^{-k}w^{n})}{\bigvee\limits_{k\in \mathfrak{I}_{n}} \Phi(e^{-k}w^{n})\,n \int\limits_{a}^{b} \psi(e^{-k}v^{n}) \frac{dv}{v}}.
     \end{align*}
     We now split the maximum over the index set:
     \begin{align*}
    \mathsf{E}_{1} & \leq \bigvee\limits_{k\in\mathscr{B}_{\tau, n}}  \left[n \int\limits_{a}^{b} \psi(e^{-k}v^{n}) \left|(h(v) - h(w))\right|\frac{dv}{v}\right]  \wedge \frac{\Phi(e^{-k}w^{n})}{\bigvee\limits_{k\in \mathfrak{I}_{n}} \Phi(e^{-k}w^{n}) n \int\limits_{a}^{b} \psi(e^{-k}v^{n}) \frac{dv}{v}} \\ &  \bigvee \bigvee\limits_{k\notin \mathscr{B}_{\tau, n}}  \left[n \int\limits_{a}^{b} \psi(e^{-k}v^{n}) \left|(h(v) - h(w))\right|\frac{dv}{v}\right]  \wedge \frac{\Phi(e^{-k}w^{n})}{\bigvee\limits_{k\in \mathfrak{I}_{n}} \Phi(e^{-k}w^{n}) n \int\limits_{a}^{b} \psi(e^{-k}v^{n}) \frac{dv}{v}} \\
      &   \leq \mathsf{E}_{1.1} \bigvee  \mathsf{E}_{1.2}.
     \end{align*}
    Since $h$ is log-continuous at $w$, for every $ \varepsilon > 0,$
    \begin{align*}
      \mathbb{E}_{1.1} %& \leq \bigvee\limits_{k\in\mathscr{B}_{\tau, n}}  \left[n \int\limits_{a}^{b} \psi(e^{-k}v^{n}) \left|(h(v) - h(w))\right|\frac{dv}{v}\right]  \wedge \frac{\Phi(e^{-k}w^{n})}{\bigvee\limits_{k\in \mathfrak{I}_{n}} \Phi(e^{-k}w^{n}) n \int\limits_{a}^{b} \psi(e^{-k}v^{n}) \frac{dv}{v}}\\
      &  \leq\bigvee\limits_{k\in\mathscr{B}_{\tau, n}}  \left[n \int\limits_{a}^{b} \psi(e^{-k}v^{n}) \; \varepsilon \;\frac{dv}{v}\right]  \wedge \frac{\Phi(e^{-k}w^{n})}{\bigvee\limits_{k\in \mathfrak{I}_{n}} \Phi(e^{-k}w^{n})n \int\limits_{a}^{b} \psi(e^{-k}v^{n}) \frac{dv}{v}}\\
      & \leq\bigvee\limits_{k\in\mathscr{B}_{\tau, n}}  \varepsilon  \wedge \frac{\Phi(e^{-k}w^{n})}{\bigvee\limits_{k\in \mathfrak{I}_{n}} \Phi(e^{-k}w^{n}) n \int\limits_{a}^{b} \psi(e^{-k}v^{n}) \frac{dv}{v}}
      \leq \bigvee\limits_{k\in\mathscr{B}_{\tau, n}} \epsilon \land 1
       \leq  \epsilon.   
    \end{align*}
     For $\mathsf{E}_{1.2}$, noting that  $a\land b \leq b$ for all $a,b >0$, we have
     \begin{align*}
     \mathsf{E}_{1.2}% & = \bigvee\limits_{k\notin \mathscr{B}_{\tau, n}}  \left[n \int\limits_{a}^{b} \psi(e^{-k}v^{n}) \left|(h(u) -h(z))\right|\frac{dv}{v}\right]  \wedge \frac{\Phi(e^{-k}w^{n})}{\bigvee\limits_{k\in \mathfrak{I}_{n}} \Phi(e^{-k}w^{n}) n \int\limits_{a}^{b} \psi(e^{-k}v^{n}) \frac{dv}{v}}\\
     &\leq \bigvee\limits_{k\notin \mathscr{B}_{\tau, n}}  \frac{\Phi(e^{-k}w^{n})}{\bigvee\limits_{k\in \mathfrak{I}_{n}} \Phi(e^{-k}w^{n}) n \int\limits_{a}^{b} \psi(e^{-k}v^{n}) \frac{dv}{v}}\\
     &\leq  \frac{1}{K\vartheta_{w}} \bigvee\limits_{k\notin \mathscr{B}_{\tau, n}}\Phi(e^{-k}w^{n})
    \leq  \frac{c n^{-\nu} }{K\vartheta_{w}}
   \leq  \epsilon
     \end{align*}
     for sufficiently large $n\in \mathbb{N}$, where $\nu$ is as given in Lemma \ref{lma2}.
     
   Finally, consider $\mathsf{E}_{2}$: 
    \begin{align*}
       \mathsf{E}_{2} = \left|\bigvee\limits_{k\in \mathfrak{I}_{n}}   \left[ n \int\limits_{a}^{b} \psi(e^{-k}v^{n}) h(w)\frac{dv}{v}\right] \wedge \frac{\Phi(e^{-k}w^{n})}{\bigvee\limits_{k\in \mathfrak{I}_{n}} \Phi(e^{-k}w^{n}) n \int\limits_{a}^{b} \psi(e^{-k}v^{n}) \frac{dv}{v}} - h(w)\right|.\\ 
    \end{align*}
    Substituting $e^{-k}v^{n}=t$ and using Lemmas \ref{lm6} and \ref{lm7}, we get
    \begin{align*}
       \mathsf{E}_{2} & = \left|\bigvee\limits_{k\in \mathfrak{I}_{n}}  h(w)\wedge \frac{\Phi(e^{-k}w^{n})}{\bigvee\limits_{k\in \mathfrak{I}_{n}} \Phi(e^{-k}w^{n}) n \int\limits_{a}^{b} \psi(e^{-k}v^{n}) \frac{dv}{v}} - h(w)\right| \\
       &  =\left|\bigvee\limits_{k\in \mathfrak{I}_{n}}  (h(w)\wedge 1)\wedge \frac{\Phi(e^{-k}w^{n})}{\bigvee\limits_{k\in \mathfrak{I}_{n}} \Phi(e^{-k}w^{n}) n \int\limits_{a}^{b} \psi(e^{-k}v^{n}) \frac{dv}{v}} -  h(w)\right|\\ 
       & \leq\left|h(w)\wedge \bigvee\limits_{k\in \mathfrak{I}_{n}}  1\wedge \frac{\Phi(e^{-k}w^{n})}{\bigvee\limits_{k\in \mathfrak{I}_{n}} \Phi(e^{-k}w^{n}) n \int\limits_{a}^{b} \psi(e^{-k}v^{n})\frac{dv}{v}} - h(w)\right|\\ 
      &  \leq\left| h(w)\wedge \bigvee\limits_{k\in \mathfrak{I}_{n}} \frac{\Phi(e^{-k}w^{n})}{\bigvee\limits_{k\in \mathfrak{I}_{n}} \Phi(e^{-k}w^{n}) n \int\limits_{a}^{b} \psi(e^{-k}v^{n}) \frac{dv}{v}} -  h(w)\right|\\ 
       &  \leq\left| h(w) \land 1  - h(w)\right|= 0.
    \end{align*}
    Hence, combining the above estimates, we obtain
\begin{align*} 
 \lim\limits_{n\to \infty} \mathscr{D}_{n,\Phi, \psi}^{m}(h) (w) = h(w),
 \end{align*}
 and if $h \in \mathscr{U}_{bl}(\mathscr{I})$, the convergence is uniform.
\end{proof}

\begin{theorem}\label{Maxminbasethm}
     Let $h \in \mathscr{U}_{bl}(\mathscr{I})$. Then, for each $\lambda > 0$, we have
\begin{align*}
\lim_{n \to \infty} I_{\zeta}\left[\lambda \left( \mathscr{D}_{n,\Phi, \psi }^{m} (h) - h \right) \right] = 0.
\end{align*}
\end{theorem}
\begin{proof}
    The proof follows directly from the arguments used in Theorem \ref{Maxthm3}, with only notational modifications corresponding to the present operators $\mathscr{D}_{n,\Phi, \psi }^{m}$. Therefore, the detailed proof is omitted.
\end{proof}
%-------------------------------------------------------------------------------------------

\begin{theorem}\label{Maxmindiffthm}
    For every $h, g\in L_{\mathfrak{h}}^\zeta(\mathscr{I})$ and for each $\lambda >0 $,  we have 
    $$I_{\zeta}\left[\lambda \left( \mathscr{D}_{n,\Phi, \psi }^{m} (h) - \mathscr{D}_{n,\Phi, \psi }^{m} (f) \right) \right] \leq \frac{\zeta(\lambda)}{K\vartheta_{w}}  \varepsilon (\lfloor{n\log b}\rfloor - \lceil{n\log a} \rceil  ) + 2\frac{\mathscr{M}_{0}(\psi)}{\| \psi\|_{1} }\,  (I_{\zeta}[\hat{\lambda} (h-f)] )^{\frac{1}{1+\beta}}. $$  
\end{theorem}
\begin{proof}
    Using the definition of the modular functional and property (c) of Lemma \ref{Maxminlma}, we have
    $$ I_{\zeta}\left[\lambda \left( \mathscr{D}_{n,\Phi, \psi }^{m} (h) - \mathscr{D}_{n,\Phi, \psi }^{m} (f) \right) \right]=\int_{a}^{b} \zeta\left(  \lambda \left| \mathscr{D}_{n,\Phi, \psi }^{m} (h) (w) - \mathscr{D}_{n,\Phi, \psi }^{m} (f) (w)\right| \right) \frac{dw}{w}.$$
    Since $\zeta$ is non-decreasing and by the positivity of the operators, it follows that
   $$ \zeta\left(  \lambda \left| \mathscr{D}_{n,\Phi, \psi }^{m} (h) (w) - \mathscr{D}_{n,\Phi, \psi }^{m} (f) (w)\right| \right)\leq \zeta\left(  \lambda  \mathscr{D}_{n,\Phi, \psi }^{m} (|h - f|) (w) \right). $$ 
Therefore,
\begin{align*}
      & I_{\zeta}\left[\lambda \left( \mathscr{D}_{n,\Phi, \psi }^{m} (h) - \mathscr{D}_{n,\Phi, \psi }^{m} (f) \right) \right] \\
       & \leq \int_{a}^{b} \zeta\left(  \lambda  \mathscr{D}_{n,\Phi, \psi }^{m}\, (|h - f|) (w) \right) \frac{dw}{w}\\
      & = \int_{a}^{b} \zeta\left(  \lambda  \bigvee\limits_{k\in \mathfrak{I}_{n}}   \left[ n \int\limits_{a}^{b} \psi(e^{-k}v^{n}) |h(v) -f(v)|\frac{dv}{v}\right] \wedge \frac{\Phi(e^{-k}w^{n})}{\bigvee\limits_{k\in \mathfrak{I}_{n}} \Phi(e^{-k}w^{n}) n \int\limits_{a}^{b} \psi(e^{-k}v^{n}) \frac{dv}{v}}\right) \frac{dw}{w}\\
      & = \int_{a}^{b} \zeta\left(   \bigvee\limits_{k\in \mathfrak{I}_{n}}   \left[\lambda n \int\limits_{a}^{b} \psi(e^{-k}v^{n}) |h(v) -f(v)|\frac{dv}{v}\right] \wedge \frac{\lambda \Phi(e^{-k}w^{n})}{\bigvee\limits_{k\in \mathfrak{I}_{n}} \Phi(e^{-k}w^{n}) n \int\limits_{a}^{b} \psi(e^{-k}v^{n}) \frac{dv}{v}}\right) \frac{dw}{w}.
    \end{align*}
    Using the monotonicity of $\zeta$ and the finiteness of $\mathfrak{I}_{n}$, we employ
     \begin{align}\label{eq2} 
     \zeta\left( \bigvee\limits_{k\in \mathscr{I}_{n}} a_{k} \right) = \bigvee\limits_{k\in \mathscr{I}_{n}} \zeta(a_{k}) \quad \text{and}\quad  \zeta\left( \bigwedge\limits_{k\in \mathscr{I}_{n}} a_{k} \right) = \bigwedge\limits_{k\in \mathscr{I}_{n}} \zeta(a_{k}).
\end{align}
\noindent
Applying \eqref{eq2} together with Lemma \ref{mainlma}, we obtain
\begin{align*}
     &  I_{\zeta}\left[\lambda \left( \mathscr{D}_{n,\Phi, \psi }^{m} (h) - \mathscr{D}_{n,\Phi, \psi }^{m} (f) \right) \right]\\& \leq \int\limits_{a}^{b} \left\{ \bigvee\limits_{k\in \mathfrak{I}_{n}}   \zeta\left(\lambda n \int\limits_{a}^{b} \psi(e^{-k}v^{n}) |h(v) -f(v)|\frac{dv}{v}\right)  \wedge \zeta\left(\frac{\lambda \Phi(e^{-k}w^{n})}{\bigvee\limits_{k\in \mathfrak{I}_{n}} \Phi(e^{-k}w^{n}) n \int\limits_{a}^{b} \psi(e^{-k}v^{n})\frac{dv}{v}}\right)\right\} \frac{dw}{w}.
\end{align*}
Using the integrability and normalization of $\psi$, we can write
     \begin{align*}
   & I_{\zeta}\left[\lambda \left( \mathscr{D}_{n,\Phi, \psi }^{m}(h) - \mathscr{D}_{n,\Phi, \psi }^{m}(f) \right)\right]\\
    &\leq \int\limits_{a}^{b} n\left\{ \bigvee\limits_{k\in \mathfrak{I}_{n}}   \zeta\left(\lambda  \int\limits_{a}^{b} \psi(e^{-k}v^{n}) |h(v) -f(v)|\frac{dv}{v}\right) \wedge \zeta\left(\frac{\lambda \Phi(e^{-k}w^{n})}{\bigvee\limits_{k\in \mathfrak{I}_{n}} \Phi(e^{-k}w^{n})n \int\limits_{a}^{b} \psi(e^{-k}v^{n})\frac{dv}{v}}\right)\right\} \frac{dw}{w}\\
    & \leq \int\limits_{a}^{b} n\left\{   \bigvee\limits_{k\in \mathfrak{I}_{n}}   \zeta\left(\lambda \| \psi\|_{1} \int\limits_{a}^{b} \frac{\psi(e^{-k}v^{n})}{\| \psi\|_{1} } |h(v) -f(v)|\frac{dv}{v}\right) \wedge \frac{\Phi(e^{-k}w^{n})}{K\vartheta_{w}} \zeta(\lambda)\right\} \frac{dw}{w}\\
       &  \leq\int\limits_{\mathbb{R}_{+}} n\left\{\bigvee\limits_{k\in \mathfrak{I}_{n}} \left( \int\limits_{a}^{b} \frac{\psi(e^{-k}v^{n})}{\| \psi\|_{1} }\zeta\left( \lambda \| \psi\|_{1}|h(v) -f(v)|\right)\frac{dv}{v}\right)  \wedge\frac{\Phi(e^{-k}w^{n})}{K\vartheta_{w}} \zeta(\lambda)\right\} \frac{dw}{w}\\
       &  \leq\int\limits_{\mathbb{R}_{+}} n\left\{\sum\limits_{k\in \mathfrak{I}_{n}} \left( \int\limits_{a}^{b} \frac{\psi(e^{-k}v^{n})}{\| \psi\|_{1} }\zeta\left( \lambda \| \psi\|_{1} |h(v) -f(v)|\right)\frac{dv}{v}\right) \wedge\frac{\Phi(e^{-k}w^{n})}{K\vartheta_{w}} \zeta(\lambda)\right\} \frac{dw}{w}.
\end{align*}
\noindent
By applying the Fubini-Tonelli theorem, we can interchange the integration and summation to obtain
    \begin{align*}
     & I_{\zeta}\left[\lambda \left( \mathscr{D}_{n,\Phi, \psi }^{m} (h) - \mathscr{D}_{n,\Phi, \psi }^{m} (f) \right) \right]\\   &  \leq    \sum\limits_{k\in \mathfrak{I}_{n}} \int\limits_{\mathbb{R}_{+}} n \left\{ \left( \int\limits_{a}^{b} \frac{\psi(e^{-k}v^{n})}{\| \psi\|_{1} }\zeta\left( \lambda \| \psi\|_{1}|h(v) -f(v)|\right)\frac{dv}{v}\right)  \wedge \frac{\Phi(e^{-k}w^{n})}{K\vartheta_{w}} \zeta(\lambda)\right\} \frac{dw}{w}\\
  & = :\sum\limits_{k\in \mathfrak{I}_{n}} \mathcal{A}_{k}.
    \end{align*}
   Denote each term by $\mathcal{A}_{k}$. For a fixed $j\in\mathfrak{I}_{n}$, substituting $y_{j}=e^{-j}w^{n}$ gives
   \begin{align*}
       \mathcal{A}_{j}&  = \int\limits_{\mathbb{R}_{+}} \left\{ \left( \int\limits_{a}^{b} \frac{\psi(e^{-k}v^{n})}{\| \psi\|_{1} }\zeta\left( \lambda \| \psi\|_{1}|h(v) -f(v)|\right)\frac{dv}{v}\right) \wedge \frac{\Phi(y_{j})}{K\vartheta_{w}} \zeta(\lambda)\right\} \frac{dy_{j}}{y_{j}}.
   \end{align*}
   Let $\varepsilon > 0$. Since $ \Phi \in L^{1}(\mathbb{R}_{+})$, there exists $n_{j} > 0 $ such that 
   $$ \int\limits_{| \log(y_{j})| > n_{j}} \Phi(y_{j}) \frac{dy_{j}}{ {y_j}} < \varepsilon.$$\\[2pt]
\noindent
Denote $ s_{j}:= (I_{\zeta}[\hat{\lambda} (h-f)] )^{\frac{-\beta}{1+\beta}} > n_{j}$. Then
   \begin{align*}
       \mathcal{A}_{j} & = \int\limits_{| \log(y_{j})| > s_{j}} \left\{ \left( \int\limits_{a}^{b} \frac{\psi(e^{-k}v^{n})}{\| \psi\|_{1} }\zeta\left( \lambda \| \psi\|_{1}|h(v) -f(v)|\right)\frac{dv}{v}\right)  \wedge \frac{\Phi(y_{j})}{K\vartheta_{w}} \zeta(\lambda)\right\} \frac{dy_{j}}{y_{j}}\\ & +  \int\limits_{| \log(y_{j})| \leq  s_{j}} \left\{ \left( \int\limits_{a}^{b} \frac{\psi(e^{-k}v^{n})}{\| \psi\|_{1} }\zeta\left( \lambda \| \psi\|_{1} |h(v) -f(v)|\right)\frac{dv}{v}\right) \wedge \frac{\Phi(y_{j})}{K\vartheta_{w}} \zeta(\lambda)\right\} \frac{dy_{j}}{y_{j}}\\
       & \leq \int\limits_{| \log(y_{j})| > s_{j}}  \frac{\Phi(y_{j})}{K\vartheta_{w}} \zeta(\lambda) \frac{dy_{j}}{y_{j}} +  \int\limits_{| \log(y_{j})| \leq  s_{j}}  \left( \int\limits_{a}^{b} \frac{\psi(e^{-k}v^{n})}{\| \psi\|_{1} }\zeta\left( \lambda \| \psi\|_{1}|h(v) -f(v)|\right)\frac{dv}{v}\right)  \frac{dy_{j}}{y_{j}}\\
      &  \leq  \frac{\zeta(\lambda)}{K\vartheta_{w}}  \varepsilon + 2 s_{j}  \int\limits_{a}^{b} \frac{\psi(e^{-k}v^{n})}{\| \psi\|_{1} }\zeta\left( \lambda \| \psi\|_{1} |h(v) -f(v)|\right)\frac{dv}{v}.
\end{align*}
Summing over $k\in \mathfrak{I}_{n}$ yields
   \begin{align*}
     &I_{\zeta}\left[\lambda \left( \mathscr{D}_{n,\Phi, \psi }^{m} (h) - \mathscr{D}_{n,\Phi, \psi }^{m} (f) \right) \right] \\  & \leq \sum\limits_{k\in \mathfrak{I}_{n}} \frac{\zeta(\lambda)}{K\vartheta_{w}}  \varepsilon + 2 s_{j}  \int\limits_{a}^{b} \frac{\psi(e^{-k}v^{n})}{\| \psi\|_{1} }\zeta\left( \lambda \| \psi\|_{1} |h(v) -f(v)|\right)\frac{dv}{v}\\
       & \leq \frac{\zeta(\lambda)}{K\vartheta_{w}}  \varepsilon (\lfloor{n\log b}\rfloor - \lceil{n\log a} \rceil  ) + 2\frac{\mathscr{M}_{0}(\psi)}{\| \psi\|_{1} } s_{j} \int\limits_{a}^{b} \zeta\left( \lambda \| \psi\|_{1}|h(v) -f(v)|\right)\frac{dv}{v}.
   \end{align*}
\noindent
Finally, for $\lambda>0$ satisfying $\lambda\|\psi\|_{1} < \hat{\lambda}$, we obtain
\begin{align*}
     &I_{\zeta}\left[\lambda \left( \mathscr{D}_{n,\Phi, \psi }^{m} (h) - \mathscr{D}_{n,\Phi, \psi }^{m} (f) \right) \right] \\% & \leq \frac{\zeta(\lambda)}{K\vartheta_{w}}  \varepsilon (\lfloor{n\log b}\rfloor - \lceil{n\log a} \rceil  ) + 2\frac{\mathscr{M}_{0}(\psi)}{\| \psi\|_{1} }\,  (I_{\zeta}[\hat{\lambda} (h-f)] )^{\frac{-\beta}{1+\beta}} \int\limits_{a}^{b} \zeta\left( \hat{\lambda}|h(v) -f(v)|\right)\frac{dv}{v}\\
      &  \leq\frac{\zeta(\lambda)}{K\vartheta_{w}}  \varepsilon (\lfloor{n\log b}\rfloor - \lceil{n\log a} \rceil  ) + 2\frac{\mathscr{M}_{0}(\psi)}{\| \psi\|_{1} }  (I_{\zeta}[\hat{\lambda} (h-f)] )^{\frac{-\beta}{1+\beta}} \;\cdot I_{\zeta}[\hat{\lambda} (h-f)] \\
       & = \frac{\zeta(\lambda)}{K\vartheta_{w}}  \varepsilon (\lfloor{n\log b}\rfloor - \lceil{n\log a} \rceil  ) + 2\frac{\mathscr{M}_{0}(\psi)}{\| \psi\|_{1} } (I_{\zeta}[\hat{\lambda} (h-f)] )^{\frac{1}{1+\beta}}.
   \end{align*}
   This completes the proof.
\end{proof}

%---------------------------------------------------------------------------------------------
\begin{theorem}
    Let $h \in L_{\mathfrak{h}}^{\zeta}(\mathscr{I})$. Then, for each $\lambda > 0$, we have
\[\lim_{n \to \infty} I_{\zeta}\left[\lambda \left( \mathscr{D}_{n,\Phi, \psi }^{m} (h) - h \right) \right] = 0.\]
\end{theorem}
\begin{proof}
    Let $\varepsilon >0$ and fix $\bar{\lambda} > 0$. By Lemma \ref{denselma}, there exists $ f \in \mathscr{U}_{bl}(\mathscr{I})$ such that  
     \begin{align}\label{eq3}
     I_{\zeta}[\hat{\lambda} (h-f)] < \varepsilon^{1+\beta}, \quad\text{ for some }~ \beta >0.
     \end{align}
    Furthermore, by Theorem \ref{Maxminbasethm}, there exists $n_{0} \in \mathbb{N}$ such that 
    \begin{align}\label{eq4}
    I_{\zeta}\left[\bar{\lambda} \left( \mathscr{D}_{n,\Phi, \psi }^{m} (f) - f \right) \right] < \varepsilon, \quad\text{ for all} ~ n \geq n_{0}.
    \end{align} 

    Now choose $\lambda > 0$ satisfying $ 3\lambda (1+ \| \psi\|_{1}) < \bar{\lambda}$. For each $ n \geq n_{0} $, by the convexity of $\zeta$, we obtain
    \begin{align*}
    I_{\zeta}\left[\lambda \left( \mathscr{D}_{n,\Phi, \psi }^{m}\, (h) - h \right) \right] %&\leq  I_{\zeta}\left[\lambda \left( \mathscr{D}_{n,\Phi, \psi }^{m} (h)- \mathscr{D}_{n,\Phi, \psi }^{m} (f) + \mathscr{D}_{n,\Phi, \psi }^{m} (f)-f +f - h \right) \right]\\
       & \leq\frac{1}{3} \left\{I_{\zeta}\left[3\lambda \left( \mathscr{D}_{n,\Phi, \psi }^{m} (h)- \mathscr{D}_{n,\Phi, \psi }^{m} (f)\right)\right] \right.\\&\qquad\quad + \left.I_{\zeta}\left[3\lambda \left(  \mathscr{D}_{n,\Phi, \psi }^{m} (f)-f\right)\right]+ I_{\zeta}\left[3\lambda \left( f-h\right)\right]\right\}.
    \end{align*}
    Applying Theorem \ref{Maxmindiffthm} together with \eqref{eq3} and \eqref{eq4}, we obtain
    \begin{align*}
    I_{\zeta}\left[\lambda \left( \mathscr{D}_{n,\Phi, \psi }^{m} (h) - h \right) \right]
       & \leq \frac{1}{3} \left\{\frac{\zeta(\lambda)}{K\vartheta_{w}}  \varepsilon (\lfloor{n\log b}\rfloor - \lceil{n\log a} \rceil  ) + 2\frac{\mathscr{M}_{0}(\psi)}{\| \psi\|_{1} }  (I_{\zeta}[\bar{\lambda} (h-f)] )^{\frac{1}{1+\beta}} \right.\\& \quad+ \left.  I_{\zeta}\left[\bar{\lambda} \left(  \mathscr{D}_{n,\Phi, \psi }^{m} (f)-f\right)\right]+I_{\zeta}\left[\bar{\lambda} \left( f-h\right)\right]\right\}\\
     &  \leq \frac{1}{3} \left\{\frac{\zeta(\lambda)}{K\vartheta_{w}}  \varepsilon (\lfloor{n\log b}\rfloor - \lceil{n\log a} \rceil  ) + 2\frac{\mathscr{M}_{0}(\psi)}{\| \psi\|_{1} }  \varepsilon + \varepsilon + \varepsilon^{1+\beta}\right\}.
    \end{align*}
    Since all terms on the right-hand side tend to zero as $\varepsilon \to 0$, the desired result follows.
\end{proof}
%----------------------------------------------------------------------------------------
\section{Examples and Graphical Results}\label{sec5}
In this section, we invstigate the approximation behavior of the proposed max-product and max-min Durrmeyer-type exponential sampling operators through both numerical evaluation and graphical visualization, employing various combinations of Mellin-type kernels. To assess the effectiveness of the operators, we consider two representative test functions with distinct characteristics.\\

The first is a smooth oscillatory function defined as 
\begin{align*}
h_{1}(w) = \frac{\log(1 + e^{0.8w\cos(2\pi w)})}{1 + \log(1 + e^{0.8w\cos(2\pi w)})},
\end{align*}
while the second is a piecewise-defined function exhibiting localized variations, given by  
\begin{align*}
h_{2}(w) =
\begin{cases}
\dfrac{(1 + \tfrac{5}{3}w)^3}{8}, & 0 \leq w < 0.6,\\[4pt]
3 - (1 + \tfrac{5}{3}w), & 0.6 \leq w < 1.2,\\[2pt]
0.4, & 1.2 \leq w < 1.8,\\[2pt]
0.8, & 1.8 \leq w < 2.4,\\[2pt]
\dfrac{\big((1 + \tfrac{5}{3}w) - 6\big)^3 + 1}{3}, & 2.4 \leq w \leq 3.
\end{cases}
\end{align*}  

\begin{example}[\textbf{Mellin $B$-Spline Kernel}]\label{example4}
 The Mellin $B$-spline kernel of order $ n \in \mathbb{N} $, for $ z > 0 $, is defined by  
\begin{align*}
B_n(w) = \frac{1}{(n-1)!} \sum_{k=0}^n (-1)^k \binom{n}{k} \left( \frac{n}{2} + \log w - k \right)_+^{n-1},
\end{align*}  
where $(v)_{+}:=\max\{v,0\}$.   
\end{example}
 The Mellin $B$-spline kernel is compactly supported on the interval $\left[e^{-\tfrac{n}{2}}, e^{\tfrac{n}{2}}\right]$ and exhibits smoothness and integrability properties essential for operators analysis. The explicit forms of $B_n(w)$ for $n=2, 3,$ and $4$ are summarized in Table~\ref{tab:Bspline}.

\begin{table}[h!]
\centering
\renewcommand{\arraystretch}{2} % for better row spacing
\begin{tabular}{|c|c|c|}
\hline
\textbf{Order $n$} & \textbf{Expression of $B_n(w)$} & \textbf{Support} \\
\hline

$n=2$ &
$ B_{2}(w)= \begin{cases}
1 + \log w, & e^{-1} \leq w \leq 1, \\[4pt]
1 - \log w, & 1 \leq w \leq e, \\[4pt]
0, & \text{otherwise},
\end{cases}$ &
$[e^{-1},\, e]$ \\
\hline

$n=3$ &
$B_{3}(w)=\begin{cases}
-\tfrac{1}{2}\bigl(\log w + \tfrac{3}{2}\bigr)^{2}, & e^{-1.5} \leq w \leq e^{-0.5}, \\[6pt]
\dfrac{3}{4} - (\log w)^{2}, & e^{-0.5} \leq w \leq e^{0.5}, \\[6pt]
-\tfrac{1}{2}\bigl(\tfrac{3}{2} - \log w \bigr)^{2}, & e^{0.5} \leq w \leq e^{1.5}, \\[6pt]
0, & \text{otherwise},
\end{cases}$ &
$[e^{-1.5},\, e^{1.5}]$ \\
\hline

$n=4$ &
$\begin{cases}
\dfrac{1}{6}(\log w + 2)^{3}, & e^{-2} \leq w \leq e^{-1}, \\[6pt]
-\dfrac{(\log w)^{3}}{2} - (\log w)^{2} + \dfrac{2}{3}, & e^{-1} \leq w \leq 1, \\[6pt]
\dfrac{(\log w)^{3}}{2} - (\log w)^{2} + \dfrac{2}{3}, & 1 \leq w \leq e, \\[6pt]
\dfrac{1}{6}(-\log w + 2)^{3}, & e \leq w \leq e^{2}, \\[6pt]
0, & \text{otherwise},
\end{cases}$ &
$[e^{-2},\, e^{2}]$ \\
\hline
\end{tabular}
\caption{Explicit forms and supports of Mellin $B$-spline kernels for $n=2,3,4$.}\label{tab:Bspline}
\end{table}

\begin{example}[\textbf{Mellin-Fej\'er Kernel}]\label{example5}
 The Mellin-Fejér kernel, parameterized by $\beta \geq 1$ and $ t \in \mathbb{R} $, is defined as 
\begin{align*}
F_\beta^t(w) = \frac{\beta}{2\pi w^{t}} \left[ \sinc\left( \frac{\beta \log \sqrt{w}}{\pi} \right) \right]^2,
\end{align*}
where $\mathrm{sinc}(x) = \frac{\sin(\pi x)}{\pi x}$.   
\end{example}
 For the present analysis, we set $\psi$ as the Mellin-Fej\'er kernel with  \(\beta = \pi\) and \( t = 0\), yielding the simplified form  
\begin{align*}
F_\pi^0(w) = \frac{1}{2} \left[ \sinc\left( \frac{\log w}{2} \right) \right]^2.
\end{align*}

\begin{example}[\textbf{Mellin-Jackson Kernel}]\label{example6}
The Mellin-Jackson kernel, a generalized version of the classical Jackson kernel in the Mellin domain, is defined for $w>0$, $\gamma \ge 1$, and $\beta \in \mathbb{N}$ as  
\begin{align*}
J_{\gamma, \beta}(w)
= d_{\gamma, \beta}
\sinc^{2\beta} \left( \frac{\log w}{2\gamma\beta\pi} \right),
\end{align*}
where $\sinc(x)=\dfrac{\sin(\pi x)}{\pi x}$, and the normalization constant $d_{\gamma, \beta}$ satisfies  
\begin{align*}
d_{\gamma, \beta}^{-1}
= \int_{0}^{\infty}
\sinc^{2\beta} \left( \frac{\log v}{2\gamma\beta\pi} \right)
\frac{dv}{v}.
\end{align*}
\end{example}
\noindent
In the present study, we adopt the modified parameter values $\gamma = 1.05$ and $\beta = 1$, resulting in  
\begin{align*}
J_{1.05,\,1}(w) = d_{1.05,\,1}\,
\sinc^{2}\!\left( \frac{\log w}{2.1\pi} \right).
\end{align*}

Finally, for the two test functions, $h_1$ and $h_2$, the numerical performance of the proposed operators has been thoroughly evaluated. The approximation characteristics are illustrated through the numerical results reported in Tables \ref{tab:error_h1B2J}-\ref{tab:error_h2B3F} and the graphical plots presented in Figures \ref{fig:MPh1B2J}-\ref{fig:h2B3Fdiff_n}, demonstrating the accuracy and convergence behavior of the operators under different kernel configurations.

Firstly, the kernel pair $(\Phi, \psi)$ consists of the Mellin $B$-spline kernel of order $2$ and the Mellin-Jackson kernel with parameters $\gamma = 1.05$ and $\beta = 1$. The proposed operators are applied to both test functions, and their approximations are illustrated for various values of $n$, highlighting the convergence behavior and shape-preserving characteristics. The numerical errors at selected points are compiled in the following tables to quantitatively assess the approximation accuracy.  
\begin{table}[htp]
\centering
\renewcommand{\arraystretch}{1.5} % row spacing
\setlength{\tabcolsep}{14pt}       % column spacing
\begin{tabular}{|c|c|cccc|}
\toprule
$n$ & Difference & $w=0.8$ & $w=1.5$ & $w=2.0$ & $w=2.5$ \\
\midrule
\multirow{2}{*}{17} 
 & $|\mathscr{D}_{17}^{M} - h_1|$ & 0.00916 & 0.11573 & 0.18534 & 0.24182 \\
 & $|\mathscr{D}_{17}^{m} - h_1|$      & 0.01972 & 0.10507 & 0.21016 & 0.22077 \\
 \midrule
\multirow{2}{*}{26} 
 & $|\mathscr{D}_{26}^{M} - h_1|$ & 0.00679 & 0.07401 & 0.10942 & 0.17724 \\
 & $|\mathscr{D}_{26}^{m} - h_1|$      & 0.01368 & 0.06771 & 0.12752 & 0.15817 \\
 \midrule
\multirow{2}{*}{35} 
 & $|\mathscr{D}_{35}^{M} - h_1|$ & 0.00497 & 0.05223 & 0.07972 & 0.12166 \\
 & $|\mathscr{D}_{35}^{m} - h_1|$      & 0.00993 & 0.04801 & 0.09501 & 0.10969 \\
 \midrule
\multirow{2}{*}{53} 
 & $|\mathscr{D}_{53}^{M} - h_1|$ & 0.00321 & 0.03536 & 0.05126 & 0.07255 \\
 & $|\mathscr{D}_{53}^{m} - h_1|$      & 0.00639 & 0.03277 & 0.06193 & 0.07035 \\
\bottomrule
\end{tabular}
\caption{Absolute errors of the operators $\mathscr{D}_{n,B_2, J_{1.05,1}}^{M}$ and 
$\mathscr{D}_{n,B_2, J_{1.05,1}}^{m}$ applied to $h_1$ at selected points for different values of $n$.}
\label{tab:error_h1B2J}
\end{table}

\begin{table}[htp]
\centering
\renewcommand{\arraystretch}{1.5} % row spacing
\setlength{\tabcolsep}{14pt}       % column spacing
\begin{tabular}{|c|c|cccc|}
\toprule
$n$ & Difference & $w=0.8$ & $w=1.5$ & $w=2.0$ & $w=2.5$ \\
\midrule
\multirow{2}{*}{17} 
 & $|\mathscr{D}_{17}^{M} - h_2|$ & 0.00701 & 0.02971 & 0.15631 & 0.30025 \\
 & $|\mathscr{D}_{17}^{m} - h_2|$ & 0.02306 & 0.01559 & 0.19140 & 0.31742 \\
\midrule
\multirow{2}{*}{26} 
 & $|\mathscr{D}_{26}^{M} - h_2|$ & 0.00626 & 0.01814 & 0.07784 & 0.25701 \\
 & $|\mathscr{D}_{26}^{m} - h_2|$ & 0.01864 & 0.00882 & 0.10245 & 0.23087 \\
\midrule
\multirow{2}{*}{35} 
 & $|\mathscr{D}_{35}^{M} - h_2|$ & 0.00399 & 0.00764 & 0.05498 & 0.23277 \\
 & $|\mathscr{D}_{35}^{m} - h_2|$ & 0.01151 & 0.00104 & 0.07527 & 0.21371 \\
\midrule
\multirow{2}{*}{53} 
 & $|\mathscr{D}_{53}^{M} - h_2|$ & 0.00334 & 0.00751 & 0.03265 & 0.13616 \\
 & $|\mathscr{D}_{53}^{m} - h_2|$ & 0.00812 & 0.00574 & 0.04654 & 0.17593 \\
\bottomrule
\end{tabular}
\caption{Absolute errors of the operators $\mathscr{D}_{n,B_2, J_{1.05,1}}^{M}$ and $\mathscr{D}_{n,B_2, J_{1.05,1}}^{m}$ applied to $h_2$ at selected points for different values of $n$.}
\label{tab:error_h2B2J}
\end{table}

\begin{figure}[h!]
    \centering
    % Adjust width as a fraction of \textwidth
    \includegraphics[width=0.75\textwidth]{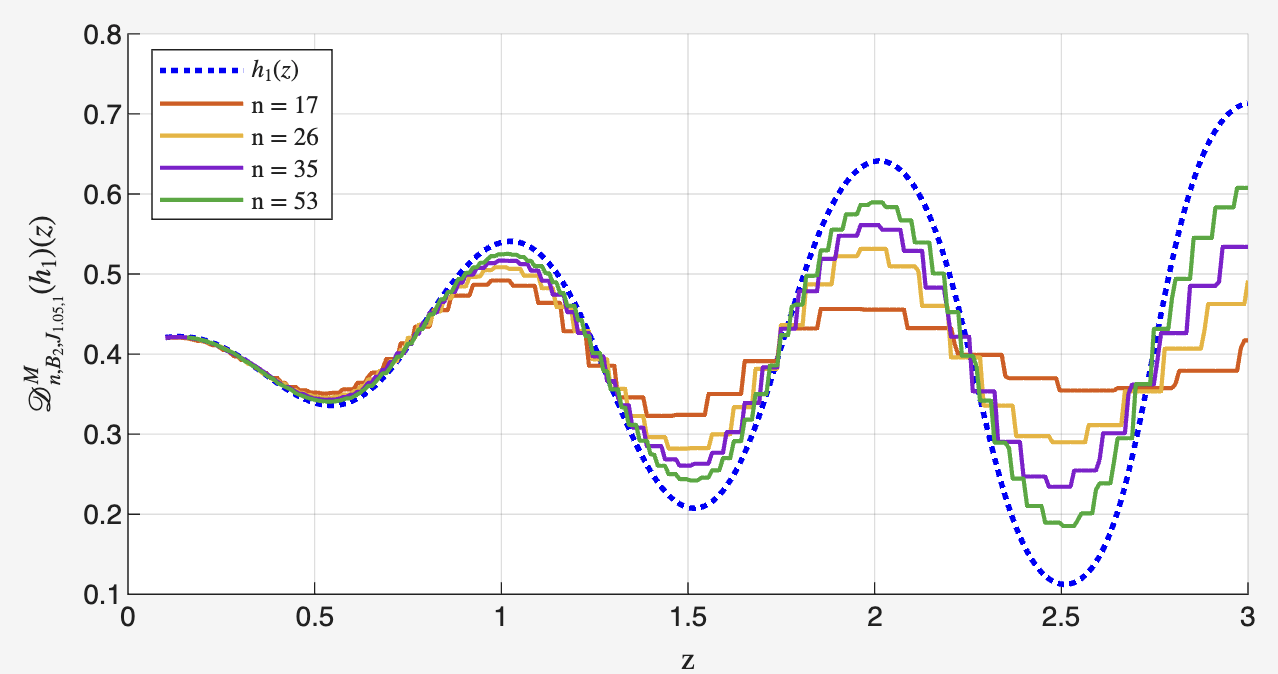}
    \caption{Approximation of $h_{1}$  by $\mathscr{D}_{n,B_{2}, J_{1.05,1} }^{M} $ .}
    \label{fig:MPh1B2J}
\end{figure}

\begin{figure}[h!]
    \centering
    % Adjust width as a fraction of \textwidth
    \includegraphics[width=0.75\textwidth]{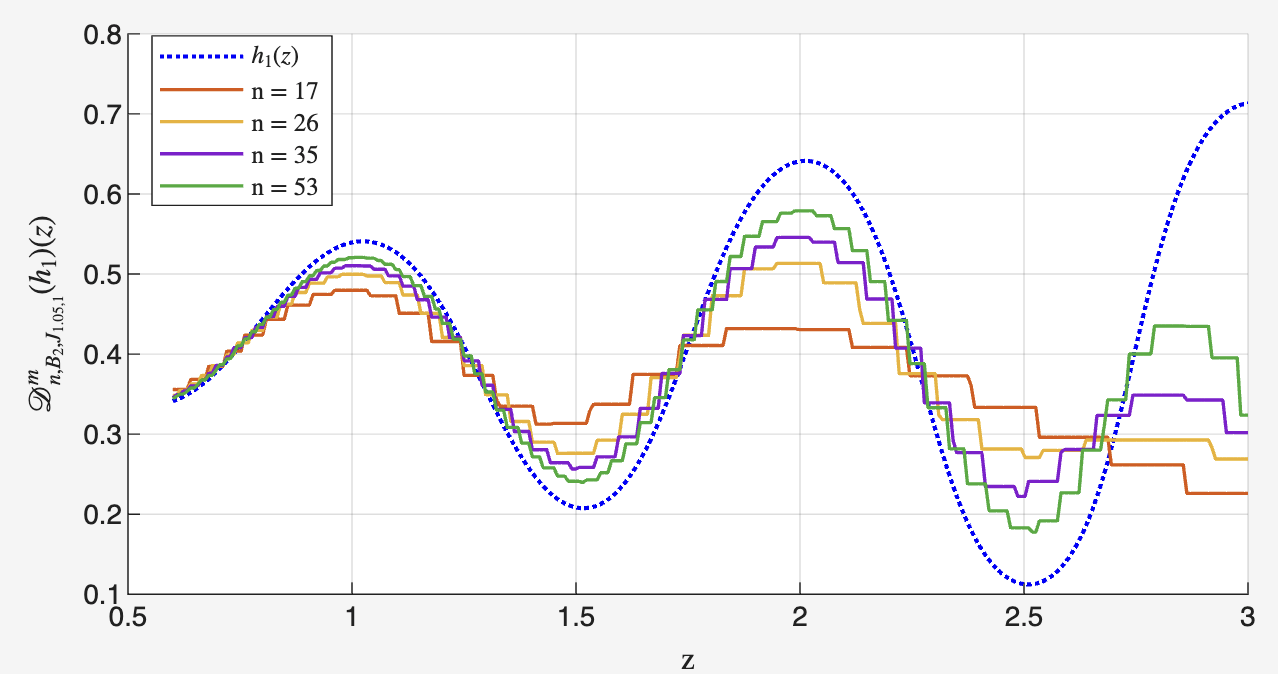}
    \caption{Approximation of $h_{1}$  by $\mathscr{D}_{n,B_{2}, J_{1.05,1} }^{m} $ .}
    \label{fig:MMh1B2J}
\end{figure}
\begin{figure}[h!]
    \centering
    % Adjust width as a fraction of \textwidth
    \includegraphics[width=0.95\textwidth]{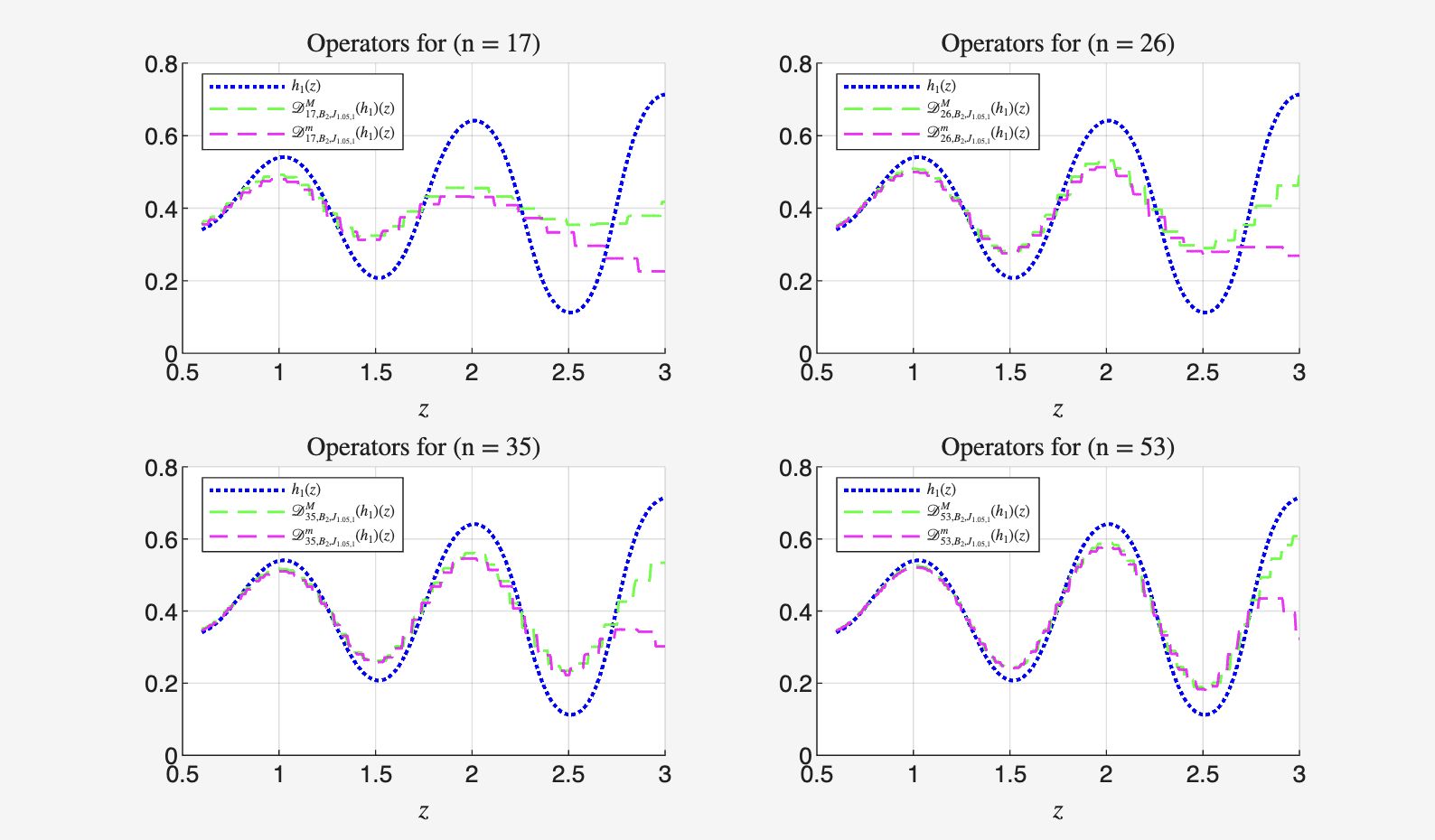}
    \caption{Behavior of the operators $\mathscr{D}_{n,B_{2}, J_{1.05,1}}^{M}$ and $\mathscr{D}_{n,B_{2}, J_{1.05,1}}^{m}$ applied to the test function $h_{1}$ for different values of $n$.}
    \label{fig:h1B2Jdiff_n}
\end{figure}
\begin{figure}[h!]
    \centering
    % Adjust width as a fraction of \textwidth
    \includegraphics[width=0.75\textwidth]{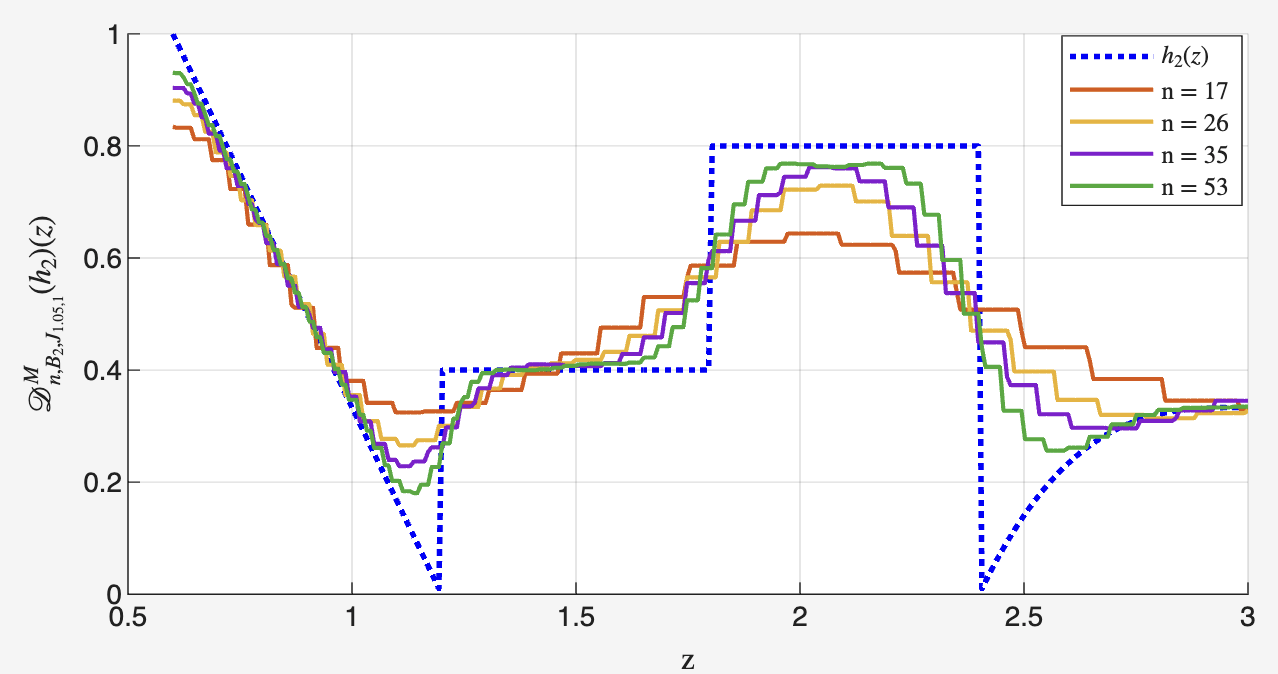}
    \caption{Approximation of $h_{2}$  by $\mathscr{D}_{n,B_{2}, J_{1.05,1} }^{M} $ .}
    \label{fig:MPh2B2J}
\end{figure}

\begin{figure}[h!]
    \centering
    % Adjust width as a fraction of \textwidth
    \includegraphics[width=0.75\textwidth]{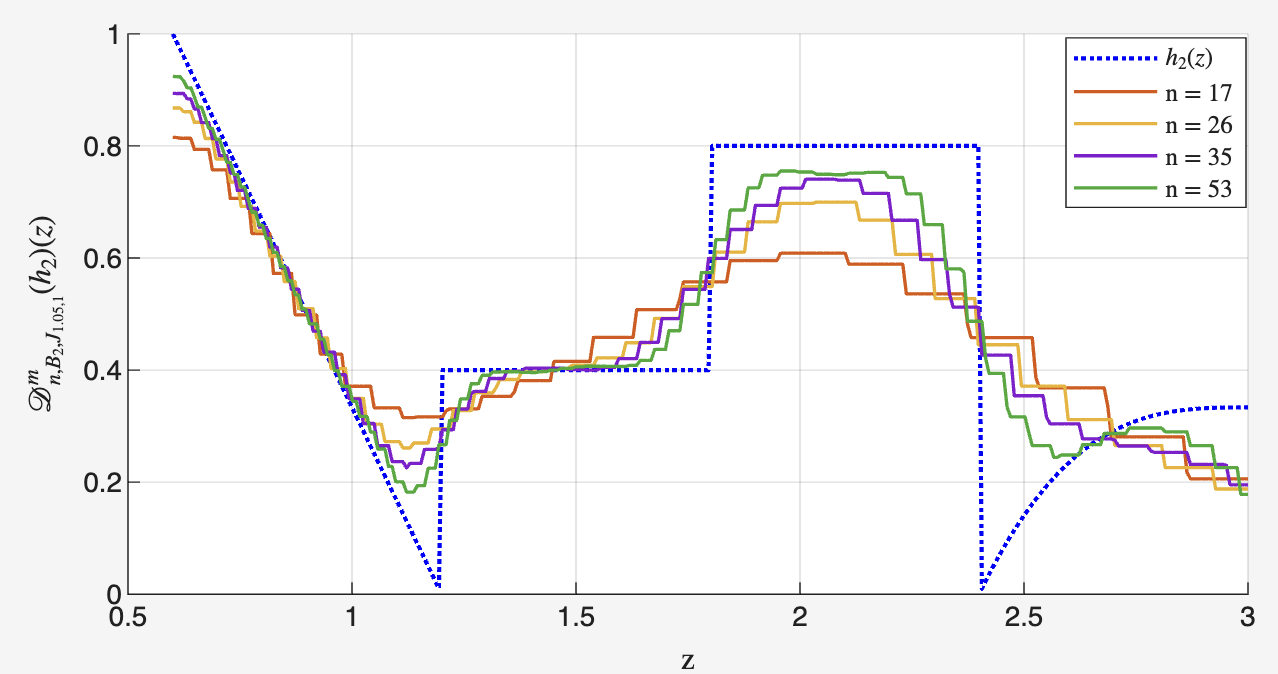}
    \caption{Approximation of $h_{2}$  by $\mathscr{D}_{n,B_{2}, J_{1.05,1} }^{m} $ .}
    \label{fig:MMh2B2J}
\end{figure}

\begin{figure}[h!]
    \centering
    % Adjust width as a fraction of \textwidth
    \includegraphics[width=0.95\textwidth]{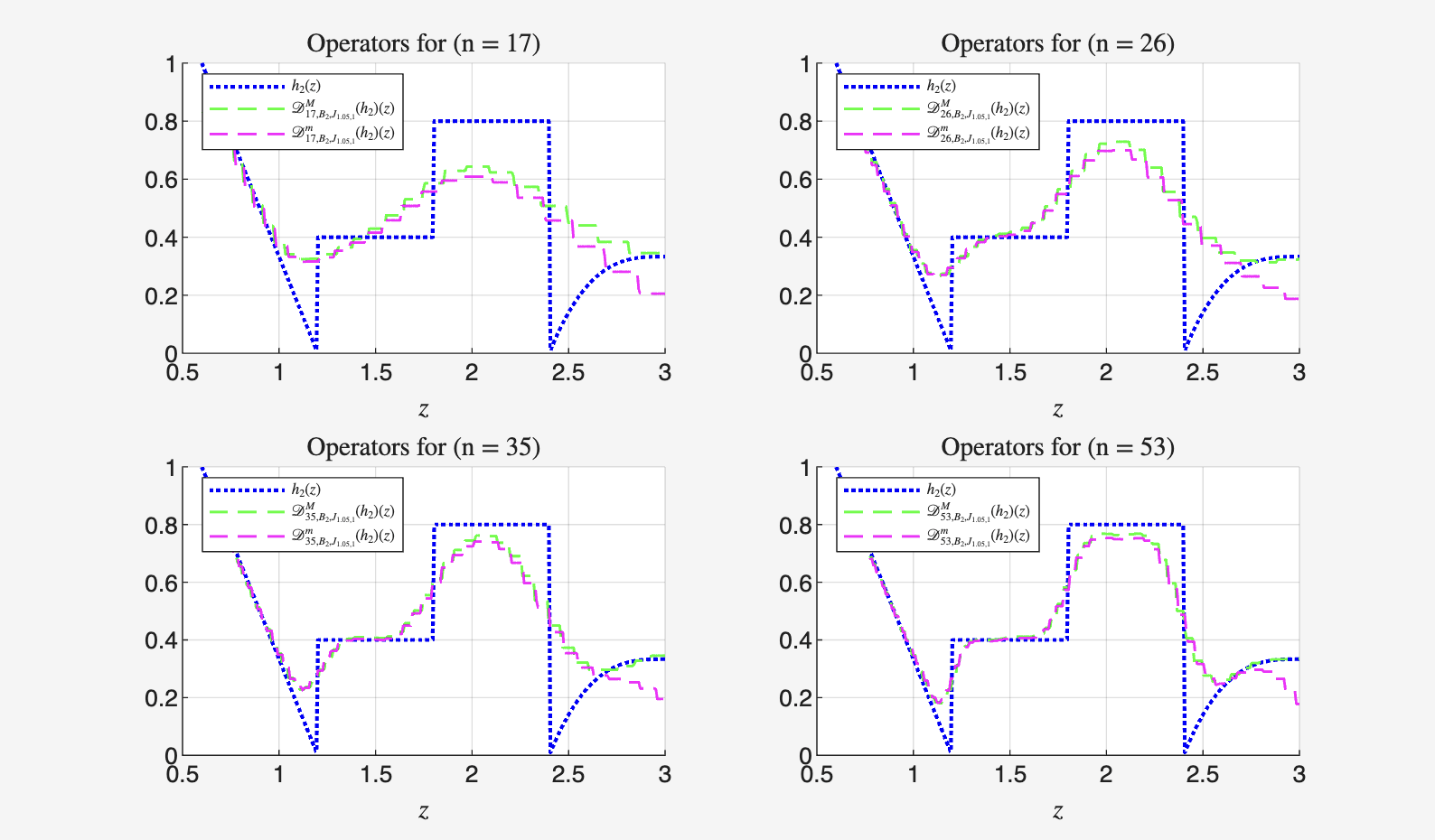}
    \caption{Behavior of the operators $\mathscr{D}_{n,B_{2}, J_{1.05,1}}^{M}$ and $\mathscr{D}_{n,B_{2}, J_{1.05,1}}^{m}$ applied to the test function $h_{2}$ for different values of $n$.}
    \label{fig:h2B2Jdiff_n}
\end{figure}
\noindent
\textbf{Observations.} ~For the $B_2$-$J_{1.05,1}$ kernel pair, the absolute errors are summarized in Tables \ref{tab:error_h1B2J} and \ref{tab:error_h2B2J}. For the oscillatory function $h_1$, both operators $\mathscr{D}_{n,B_2,J_{1.05,1}}^{M}$ and $\mathscr{D}_{n,B_2,J_{1.05,1}}^{m}$ exhibit a steady reduction in error as $n$ increases, confirming their strong convergence. For instance, at $w=2.0$, the max-product operator error decreases from $0.18534$ ($n=17$) to $0.05126$ ($n=53$), while the max-min operator error reduces from $0.21016$ to $0.06193$. Figures \ref{fig:MPh1B2J}-\ref{fig:h1B2Jdiff_n} visually substantiate these findings, showing increasingly closer alignment between the approximated and exact curves with growing $n$.

A similar convergence trend is observed for the piecewise-defined function $h_2$ (Table \ref{tab:error_h2B2J}) though with slightly slower error decay near discontinuities. The absolute error decreases from $0.30025$ ($n=17$) to $0.13616$ ($n=53$), illustrating the robustness of the proposed operators in preserving the functional shape even for non-smooth inputs. The graphical results in Figures \ref{fig:MPh2B2J}-\ref{fig:h2B2Jdiff_n} further confirm these convergence and stability characteristics.

Secondly, the Mellin $B$-spline kernel of order $3$ is combined with the Mellin-Fejér kernel with parameters $\beta = \pi$ and $t = 0$. Similar analyses are performed for both test functions, including graphical comparison of the max-product and max-min operators, along with tabulated error evaluations for varying values of $n$. This comprehensive examination provides further insights into the influence of kernel selection on the convergence rate, stability, and overall approximation performance of the proposed operators across different functional settings.

\begin{table}[htp]
\centering
\renewcommand{\arraystretch}{1.5} % row spacing
\setlength{\tabcolsep}{14pt}       % column spacing
\begin{tabular}{|c|c|cccc|}
\toprule
$n$ &  Difference & $w=0.8$ & $w=1.5$ & $w=2.0$ & $w=2.5$ \\
\midrule
\multirow{2}{*}{17} 
 & $|\mathscr{D}_{17}^{M} - h_1|$ & 0.00757 & 0.03185 & 0.04977 & 0.08111 \\
 & $|\mathscr{D}_{17}^{m} - h_1|$      & 0.01079 & 0.02929 & 0.06028 & 0.07262 \\
\midrule
\multirow{2}{*}{26} 
 & $|\mathscr{D}_{26}^{M} - h_1|$ & 0.00493 & 0.02111 & 0.03161 & 0.04387 \\
 & $|\mathscr{D}_{26}^{m} - h_1|$      & 0.00705 & 0.01948 & 0.03842 & 0.03996 \\
\midrule
\multirow{2}{*}{35} 
 & $|\mathscr{D}_{35}^{M} - h_1|$ & 0.00359 & 0.01639 & 0.02494 & 0.03287 \\
 & $|\mathscr{D}_{35}^{m} - h_1|$      & 0.00517 & 0.01521 & 0.02995 & 0.03035 \\
\midrule
\multirow{2}{*}{53} 
 & $|\mathscr{D}_{53}^{M} - h_1|$ & 0.00223 & 0.01219 & 0.01501 & 0.02202 \\
 & $|\mathscr{D}_{53}^{m} - h_1|$      & 0.00328 & 0.01144 & 0.01836 & 0.02049 \\
\bottomrule
\end{tabular}
\caption{Absolute errors of the operators 
$\mathscr{D}_{n,\,B_3, F_{\pi}^{0}}^{M}$ and 
$\mathscr{D}_{n,\,B_3, F_{\pi}^{0}}^{m}$ applied to $h_1$ 
at selected points $z$ for different values of $n$.}
\label{tab:error_h1B3F}
\end{table}

\begin{table}[htp]
\centering
\renewcommand{\arraystretch}{1.5} % row spacing
\setlength{\tabcolsep}{14pt}       % column spacing
\begin{tabular}{|c|c|cccc|}
\toprule
$n$ &  Difference & $w=0.8$ & $w=1.5$ & $w=2.0$ & $w=2.5$ \\
\midrule
\multirow{2}{*}{17} 
 & $|\mathscr{D}_{17}^{M} - h_2|$ & 0.00787 & 0.00803 & 0.03231 & 0.10952 \\
 & $|\mathscr{D}_{17}^{m} - h_2|$      & 0.00289 & 0.00367 & 0.04597 & 0.09858 \\
\midrule
\multirow{2}{*}{26} 
 & $|\mathscr{D}_{26}^{M} - h_2|$ & 0.00507 & 0.00563 & 0.02354 & 0.06554 \\
 & $|\mathscr{D}_{26}^{m} - h_2|$      & 0.00182 & 0.00275 & 0.03222 & 0.06041 \\
\midrule
\multirow{2}{*}{35} 
 & $|\mathscr{D}_{35}^{M} - h_2|$ & 0.00354 & 0.00341 & 0.01572 & 0.03771 \\
 & $|\mathscr{D}_{35}^{m} - h_2|$      & 0.00114 & 0.00128 & 0.02210 & 0.03461 \\
\midrule
\multirow{2}{*}{53} 
 & $|\mathscr{D}_{53}^{M} - h_2|$ & 0.00191 & 0.00205 & 0.00971 & 0.05097 \\
 & $|\mathscr{D}_{53}^{m} - h_2|$      & 0.00032 & 0.00068 & 0.01395 & 0.04881 \\
\bottomrule
\end{tabular}
\caption{Absolute errors of the operators 
$\mathscr{D}_{n,\,B_3, F_{\pi}^{0}}^{M}$ and 
$\mathscr{D}_{n,\,B_3, F_{\pi}^{0}}^{m}$ applied to $h_2$ 
at selected points $z$ for different values of $n$.}
\label{tab:error_h2B3F}
\end{table}

\begin{figure}[h!]
    \centering
    % Adjust width as a fraction of \textwidth
    \includegraphics[width=0.75\textwidth]{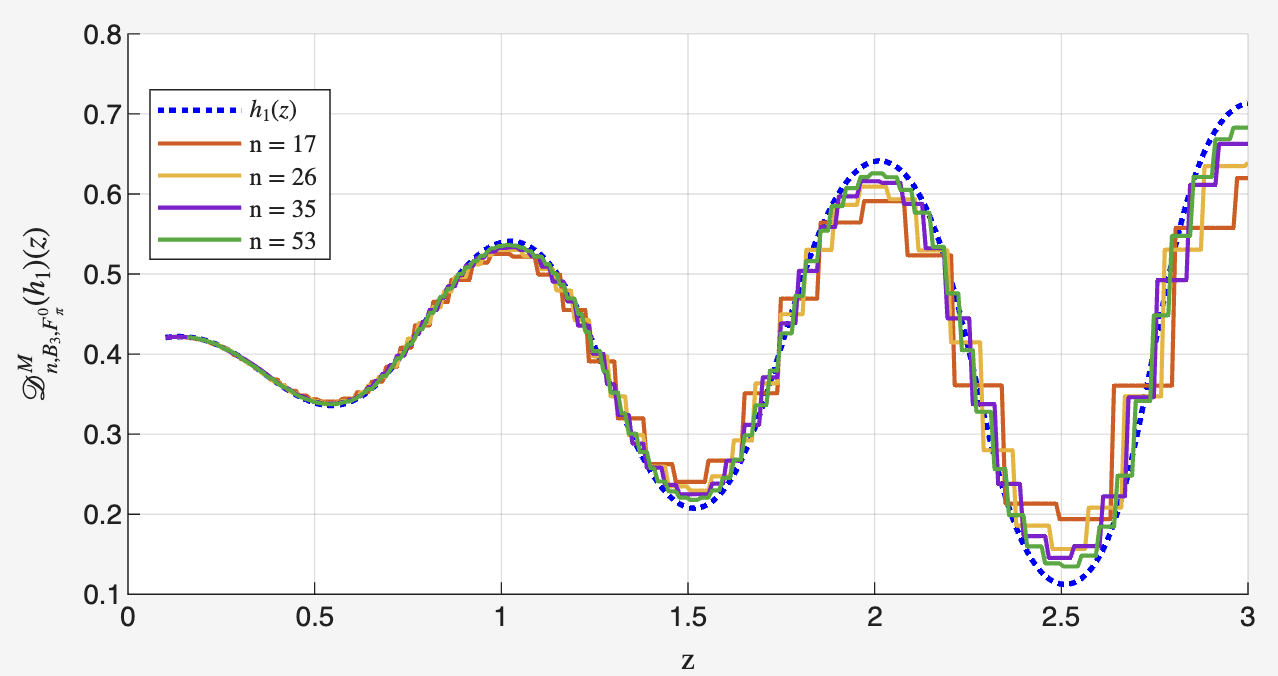}
    \caption{Approximation of $h_{1}$  by $\mathscr{D}_{n, B_{3},F_{\pi}^{0} }^{M} $ .}
    \label{fig:MPh1B3F}
\end{figure}

\begin{figure}[h!]
    \centering
    % Adjust width as a fraction of \textwidth
    \includegraphics[width=0.75\textwidth]{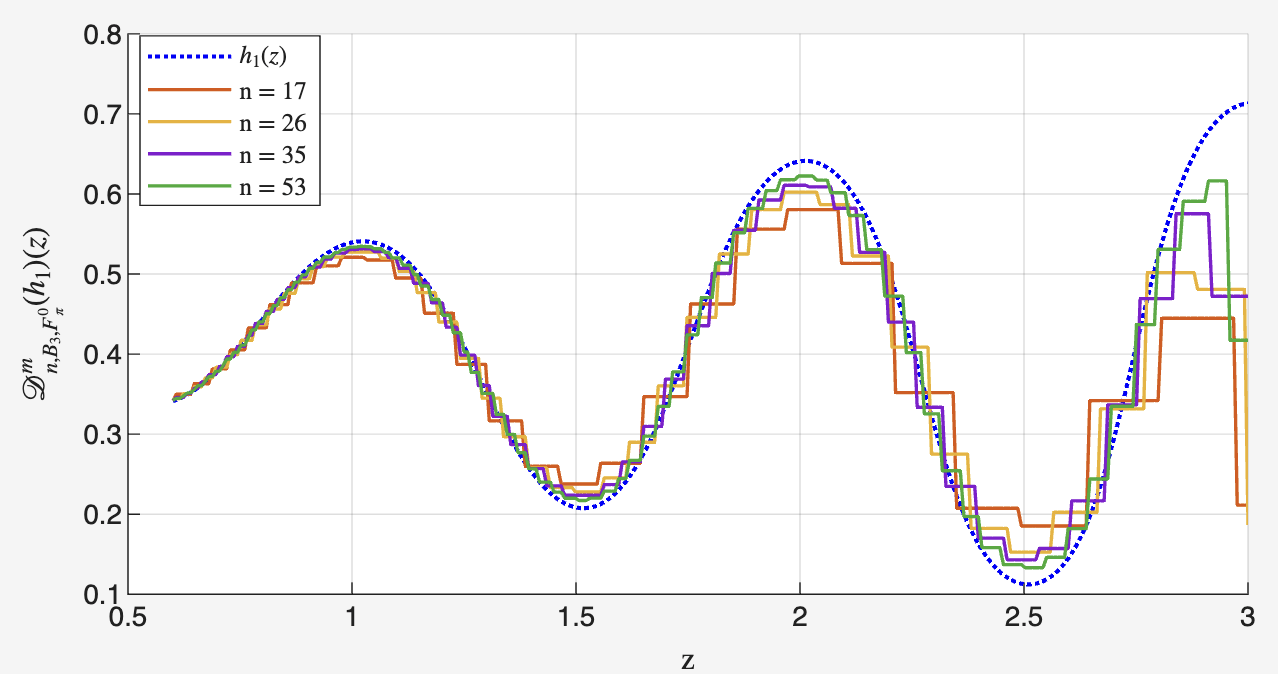}
    \caption{Approximation of $h_{1}$  by $\mathscr{D}_{n, B_{3}, F_{\pi}^{0} }^{m} $ .}
    \label{fig:MMh1B3F}
\end{figure}

\begin{figure}[h!]
    \centering
    % Adjust width as a fraction of \textwidth
    \includegraphics[width=0.95\textwidth]{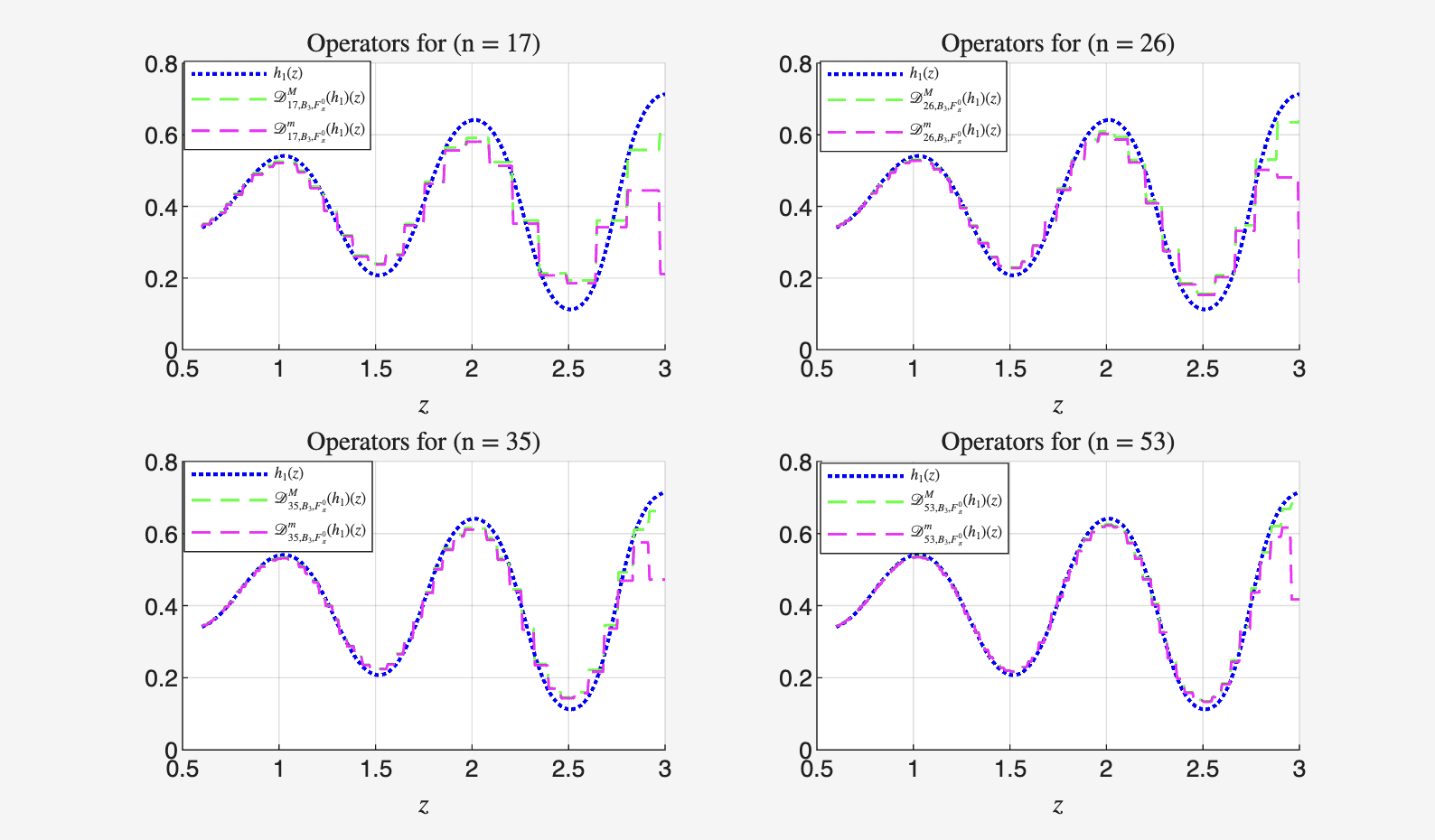}
    \caption{Behavior of the operators $\mathscr{D}_{n, B_{3}, F_{\pi}^{0} }^{M} $ and $\mathscr{D}_{n, B_{3}, F_{\pi}^{0} }^{m} $ applied to the test function $h_{1}$ for different values of $n$}
    \label{fig:h1B3Fdiff_n}
\end{figure}

\begin{figure}[h!]
    \centering
    % Adjust width as a fraction of \textwidth
    \includegraphics[width=0.75\textwidth]{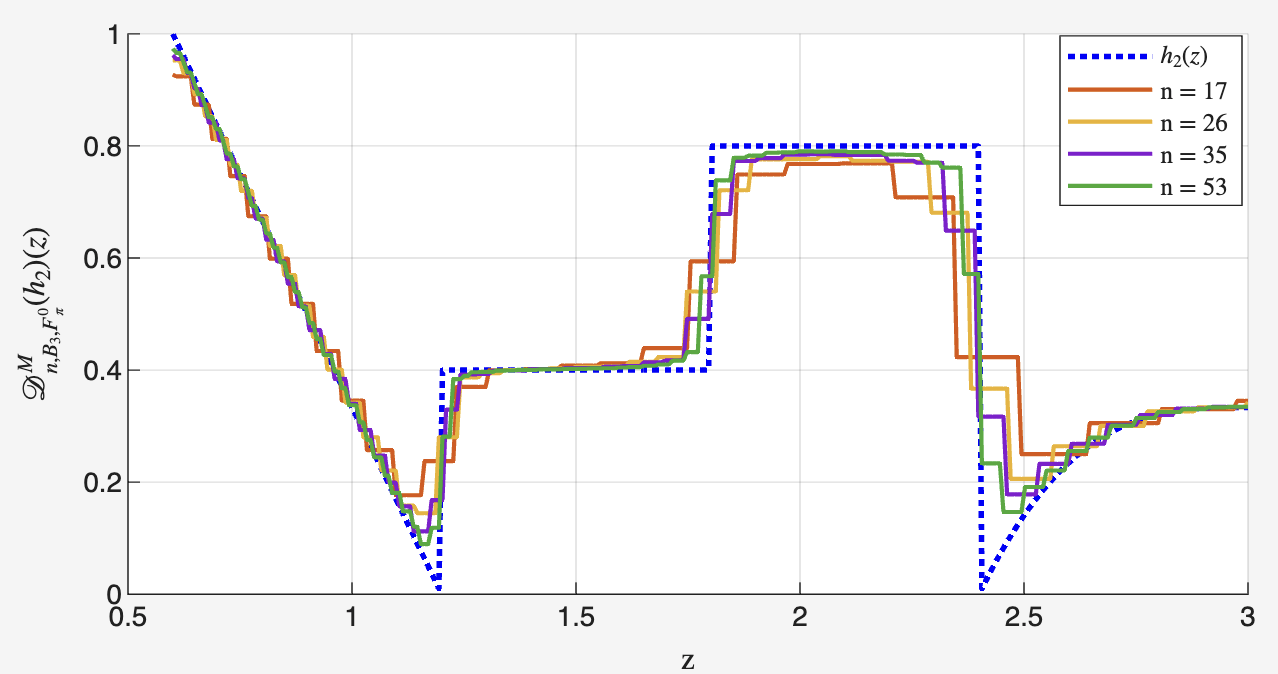}
    \caption{Approximation of $h_{2}$  by $\mathscr{D}_{n, B_{3}, F_{\pi}^{0} }^{M} $ .}
    \label{fig:MPh2B3F}
\end{figure}

\begin{figure}[h!]
    \centering
    % Adjust width as a fraction of \textwidth
    \includegraphics[width=0.75\textwidth]{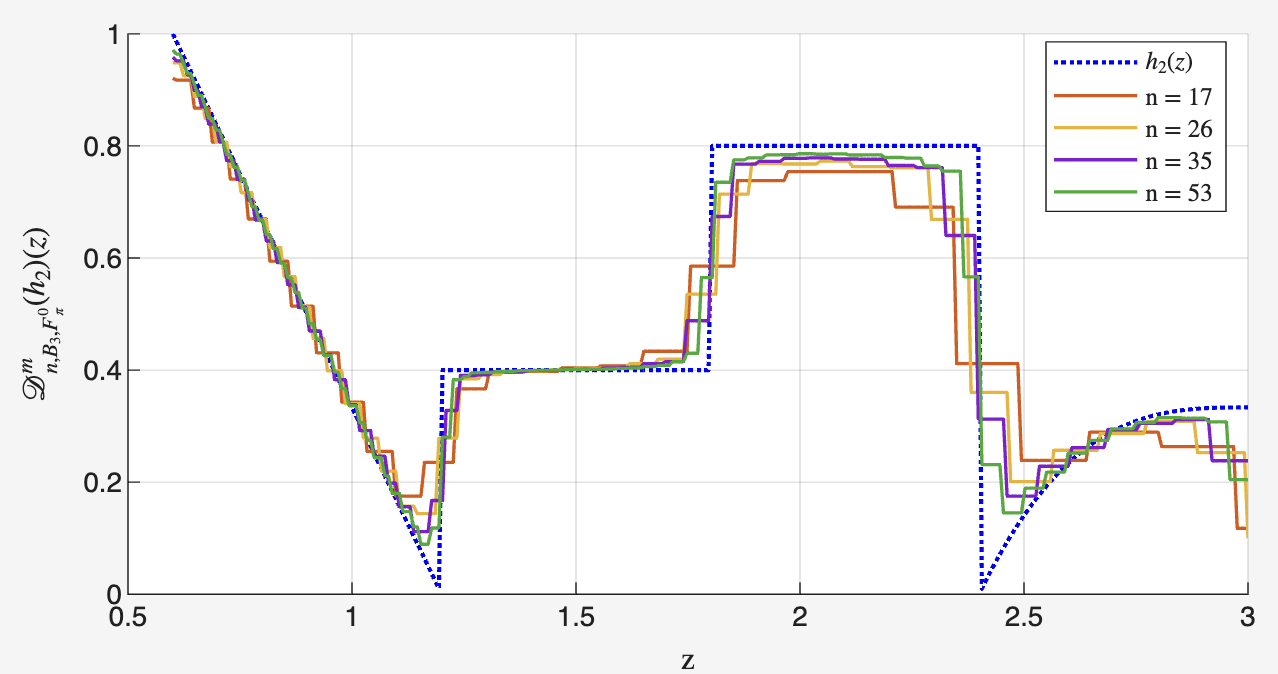}
    \caption{Approximation of $h_{2}$  by $\mathscr{D}_{n, B_{3}, F_{\pi}^{0} }^{m} $ .}
    \label{fig:MMh2B3F}
\end{figure}

\begin{figure}[h!]
    \centering
    % Adjust width as a fraction of \textwidth
    \includegraphics[width=0.95\textwidth]{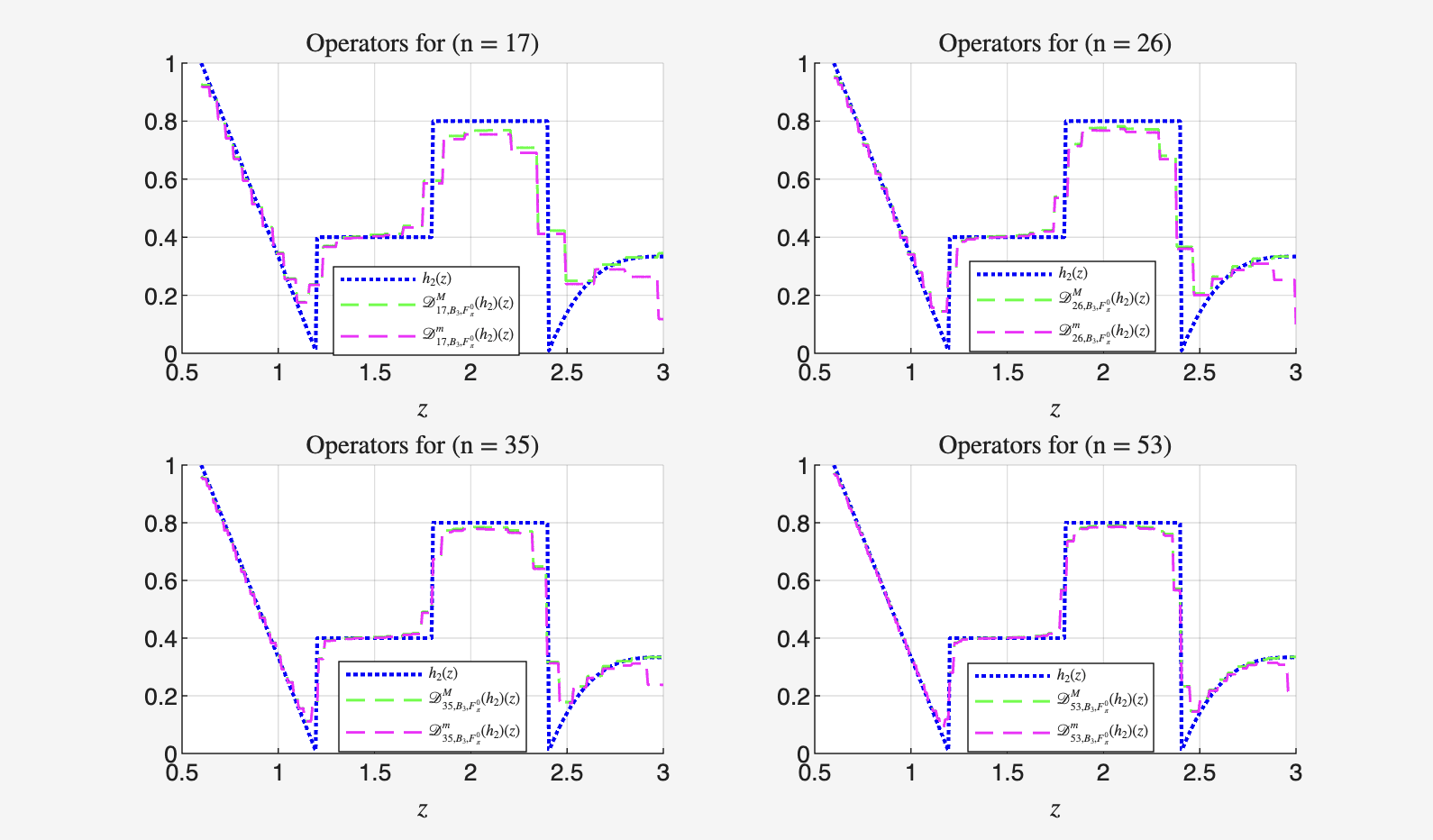}
    \caption{Behavior of the operators $\mathscr{D}_{n, B_{3}, F_{\pi}^{0} }^{M} $ and $\mathscr{D}_{n, B_{3}, F_{\pi}^{0} }^{m} $ applied to the test function $h_{2}$ for different values of $n$}
    \label{fig:h2B3Fdiff_n}
\end{figure}

\textbf{Observations:} ~Tables \ref{tab:error_h1B3F} and \ref{tab:error_h2B3F} present the absolute errors corresponding to the $B_3$-$F_{\pi}^{0}$ kernel pair. Compared with the Jackson-based configuration, this higher-order kernel combination achieves faster convergence and smoother approximations. For example, at $w=2.0$, the max-product operator error decreases from $0.04977$ ($n=17$) to $0.01501$ ($n=53$), while the max-min operator error reduces from $0.06028$ to $0.01836$. The graphical results in Figures \ref{fig:MPh1B3F}-\ref{fig:h1B3Fdiff_n} clearly illustrate this improvement, showing progressively smoother and more accurate reconstructions as $n$ increases.

For the piecewise-defined function $h_2$, both operators yield very small absolute errors (see Table \ref{tab:error_h2B3F}). For instance, at $w=0.8$, the max-product and max-min errors decrease from $0.00787$ and $0.00289$ to $0.00191$ and $0.00032$, respectively, with increasing $n$. Figures \ref{fig:MPh2B3F}-\ref{fig:h2B3Fdiff_n} confirm the strong and stable convergence of the proposed operators, with minimal oscillatory behavior and excellent shape preservation even near non-smooth regions.
\section*{Comparative Discussion}
Across all kernel configurations, the max-product operators consistently outperform their max-min counterparts in smooth regions, owing to their inherent multiplicative smoothing behavior that effectively suppresses oscillations and preserves the overall curvature of the target functions. In contrast, the max-min operators exhibit superior stability near discontinuities, where they efficiently mitigate the influence of local extrema. The steady error decay documented in all numerical tables substantiates the theoretical convergence results established in the preceding sections. Moreover, the use of higher-order kernels leads to smoother approximations and faster convergence, underscoring their enhanced efficiency in representing both continuous and piecewise-defined functions.

\section{Conclusion}\label{sec:conclusion}
This work introduced two nonlinear approximation schemes the max–product and max–min Durrmeyer-type exponential sampling operators formulated within the framework of Orlicz spaces. Also, we established convergence in pointwise and uniform senses and analyzed the operators’ approximation behavior using exponential-type Mellin kernels such as the Mellin $B$-spline, Mellin–Jackson, and Mellin–Fejér kernels,.

Theoretical and numerical results confirmed that both operators exhibit stable convergence, with higher-order kernels ensuring smoother and faster approximation. The max–product operators perform better in smooth regions, while the max–min variants demonstrate superior stability near discontinuities.

Future extensions may explore multidimensional forms, stochastic convergence, and adaptive kernel strategies to enhance accuracy in non-smooth and high-frequency domains.
\section*{Acknowledgements}

We sincerely thank the reviewers for their valuable suggestions that improved the paper's quality.

\section*{Declarations}
\begin{itemize}
\item {\bf Ethical Approval}: Not applicable.
\item {\bf Funding}: No funding or support was received for this research.
\item {\bf Conflict of interest/Competing interests}: The authors have no competing interests.
\item {\bf Data availability}: All relevant data from this study are included in the article.
\end{itemize}
%++++++++++++++++++++++++++++++++++++++++++++++++++++++++++++++++++++++++++++++++++++++++++++++++++++++++++++++++++++++++++++++++++++++++++++++++++++++++++++++++++++++++++++++++++++++++++++++++++++++++




\end{document}